\pdfoutput=1
\documentclass[final]{dmtcs-episciences}

\usepackage[utf8]{inputenc}
\usepackage[T1]{fontenc}

\usepackage{amsthm, amsmath, amsfonts}
\usepackage{enumerate}

\usepackage{stmaryrd}

\graphicspath{{figures/}}
\usepackage{tikz}
\usetikzlibrary{arrows,backgrounds,calc}
\usetikzlibrary{math,arrows,backgrounds}
\usetikzlibrary{decorations.pathreplacing}
\usetikzlibrary{fit}
\usepackage{pgfplots}
\pgfplotsset{compat=1.12}
\usepackage{pgfplotstable}
\usepackage{mathtools}

\usepackage[labelfont=normalfont, textfont=normalfont]{subcaption}

\usepackage{hyperref}

\allowdisplaybreaks

\newtheorem{theorem}{Theorem}
\newtheorem{proposition}[theorem]{Proposition}
\newtheorem{corollary}[theorem]{Corollary}
\newtheorem{lemma}[theorem]{Lemma}
\newtheorem{observation}[theorem]{Observation}
\newtheorem{claim}[theorem]{Claim}
\newtheorem*{claim*}{Claim}

\newtheorem*{theorem*}{Theorem}
\newtheorem*{proposition*}{Proposition}
\newtheorem*{corollary*}{Corollary}
\newtheorem*{lemma*}{Lemma}

\theoremstyle{remark}
\newtheorem{remark}{Remark}

\theoremstyle{definition}
\newtheorem{definition}[theorem]{Definition}

\newcommand{\DLMF}[2]{\cite[\href{http://dlmf.nist.gov/#1.E#2}{#1.#2}]{NIST:DLMF:v1.0.26}}
\newcommand{\OEIS}[1]{\text{\href{https://oeis.org/#1}{{\small \tt #1}}}}

\DeclareSymbolFont{cmcal}{OMS}{cmsy}{m}{n}
\SetSymbolFont{cmcal}{bold}{OMS}{cmsy}{b}{n}
\DeclareSymbolFontAlphabet{\mathcal}{cmcal}

\newcommand{\N}[0]{\mathbb{N}}

\newcommand{\E}[0]{\mathbb{E}}
\newcommand{\V}[0]{\mathbb{V}}

\newcommand{\EL}[0]{\mathcal{E}}
\newcommand{\MA}[0]{\mathcal{M}}

\DeclarePairedDelimiter{\iverson}{\llbracket}{\rrbracket}

\def\down{\mathbin{
\begin{tikzpicture}[scale = 0.25]
\draw[-] (0, 0) to (1, -1);
\end{tikzpicture}
}}

\def\up{\mathbin{
\begin{tikzpicture}[scale = 0.25]
\draw[-] (0, 0) to (0.3, 1);
\end{tikzpicture}
}}

\newcommand\textboxa[1]{%
  \parbox{.33\textwidth}{#1}%
}
\newcommand\textboxb[1]{%
  \parbox{.24\textwidth}{#1}%
}
\newcommand\textboxc[1]{%
  \parbox{.42\textwidth}{\vspace{0.5ex}\raggedleft #1\vspace{0.5ex}}%
}

\definecolor{dgreen}{rgb}{0.0, 0.5, 0.2}

\newcommand{\drawlatticepath}[1]{
  \tikzmath{
      \x = 0;
      \y = 0;
      for \step in {#1}{
        {
          \draw[thin, ->] (\x, \y) -- (\x + 1, \y + \step);
        };
        \x = \x + 1;
        \y = \y + \step;
      };
    };
}

\title[Down-steps in generalized Dyck paths]{Down-step statistics in generalized Dyck paths}

\author{
  Andrei Asinowski\affiliationmark{1}
  \and Benjamin Hackl\affiliationmark{1,2}
  \and Sarah J. Selkirk\affiliationmark{1}\thanks{
    A.~Asinowski, B.~Hackl, and S.~J.~Selkirk are supported by the
    Austrian Science Fund (FWF): P~28466-N35, \emph{Analytic Combinatorics: Digits, Automata and Trees}, S.~J.~Selkirk 
    is supported by Austrian Science Fund (FWF): DOC 78, and A.~Asinowski is supported by Austrian Science Fund (FWF): P~32731, 
    \emph{Generic Rectangulations: Enumerative and Structural Aspects}. 
  }
}

\affiliation{
  Institut für Mathematik, Alpen-Adria-Universität Klagenfurt, Austria\\
  Department of Mathematics, Uppsala University, Sweden
}


\received{2021-02-08}
\revised{2021-10-14, 2022-04-19}
\accepted{2022-04-21}

\begin{document}
\publicationdetails{24}{2022}{1}{17}{7163}

\maketitle

\begin{abstract}
The number of down-steps between pairs of up-steps in $k_t$-Dyck paths,
a generalization of Dyck paths consisting of steps $\{(1, k), (1, -1)\}$ such that the path
  stays (weakly) above the line $y=-t$, is studied. Results are proved
  bijectively and by means of generating functions, and lead to several interesting
  identities as well as links to other combinatorial structures. In
  particular, there is a connection between $k_t$-Dyck paths and
  perforation patterns for punctured convolutional codes (binary matrices) used in coding theory.
  Surprisingly, upon restriction to usual Dyck paths this
  yields a new combinatorial interpretation of Catalan numbers.
\end{abstract}
\maketitle

\section{Introduction}\label{sec:introduction}

This article presents a comprehensive investigation of the number of
down-steps between pairs of up-steps in a family of generalized Dyck paths
known as $k_t$-Dyck paths. In particular, we obtain an explicit formula
for this parameter, and explore its asymptotic growth.
Throughout, we further establish properties
of $k_t$-Dyck paths and their use in finding combinatorial interpretations
of identities.
To begin, we define these paths
and associated objects, and provide background and motivation for studying this parameter.
\begin{definition}[$k$-Dyck path]
  Let $k$ be a positive integer.  A \emph{$k$-Dyck path} is a lattice path that consists of
  \emph{up-steps} $(1,k)$ and \emph{down-steps} $(1,-1)$, starts at $(0,0)$,
  stays weakly above the line $y = 0$ and ends on the line $y = 0$.
\end{definition}
Notice that if a $k$-Dyck path has $n$ up-steps, then it has $kn$ down-steps, and thus
has \emph{length} $(k+1)n$. The family of $k$-Dyck paths in this form were introduced by Hilton and Pedersen
\cite{Hilton-Pedersen:1991:catalan-generalization} and have since been the
topic of several studies involving statistics in these paths
and their relations to other mathematical objects
(see, for example, \cite{Cameron:2016:Returns, Hackl-Heuberger-Prodinger:ta:lukasiewicz-ascents,
  Josuat-Kim:2016:Dyck-tilings, Xin-Zhang:2019:sweep-maps}).
However, note that equivalent lattice path families, such as the family of
paths consisting of horizontal and vertical unit steps staying weakly below
the line $y = kx$, have been under consideration for much longer --- see, for example,
\cite[Section~1.4]{Mohanty:1979:lattic} where this particular family is discussed,
including references to older literature. More information can also be found
in the recent survey article~\cite{Krattenthaler:2015:lattic} by Krattenthaler.

In particular, $k$-Dyck paths are a special case of \L{}ukasiewicz paths.
A \emph{\L{}ukasiewicz path} of length $n$ is a lattice path
consisting of steps $(1, i)$ where $i \in \{-1, 0, 1, 2, \ldots\}$,
that starts at $(0, 0)$, stays weakly above the line $y=0$,
and terminates at $(n, 0)$.
There is a well-known natural bijective mapping between \L{}ukasiewicz paths and
plane trees~\cite[Chapter~11]{Lothaire:1997:comb-words}
(see also~\cite[Example~3]{Banderier-Flajolet:2002:lat-path}).
Under this mapping, $k$-Dyck paths correspond to $(k+1)$-ary trees, and thus $k$-Dyck paths of length $(k+1)n$ are enumerated by
\emph{Fuss--Catalan numbers}
(see~\cite[Example I.14]{Flajolet-Sedgewick:ta:analy}) which are given by
\begin{equation}\label{eq:fuss-catalan}
  C_{n}^{(k)} = \frac{1}{kn+1} \binom{(k+1)n}{n}.
\end{equation}

In \cite{Gu-Prodinger-Wagner:2010:k-plane-trees}, Gu, Prodinger, and Wagner
introduced a class of $k$-coloured plane trees called \emph{$k$-plane trees},
which are plane trees where each vertex is labeled $i$ with $i \in \{1, 2, \ldots, k\}$ and the sum
of labels along any edge is at most $k+1$ (with the root labeled $k$).
They construct
a bijection between this class and $k$-Dyck paths as well as $(k+1)$-ary trees.
Furthermore, they note that removing the condition that the root is labeled $k$ gives a bijection
with generalized $k$-Dyck paths
which can go below the $x$-axis until some fixed negative height.
There they
mention that only few enumeration problems lead to the
numbers corresponding to the total number of generalized $k$-Dyck paths,
\begin{equation*}
  \frac{k}{n+1}\binom{(k+1)n}{n}.
\end{equation*}
We provide a new combinatorial interpretation for these numbers in \eqref{eq:downstep-recurrence:elevated}.

Recently, this generalization of $k$-Dyck paths
was further developed and enumerated by Selkirk
in~\cite{Selkirk:2019:MSc}, where such paths were given the name $k_t$-Dyck paths.
It is this family of paths that will be the main object of interest in our article.

\begin{definition}[$k_t$-Dyck paths]\label{def:k_t-Dyck}
  Let $k$ and $t$ be integers with $k > 0$ and $t \geq 0$.  A \emph{$k_{t}$-Dyck path} is a lattice path
  that consists of up-steps $(1, k)$ and down-steps $(1, -1)$, starts at $(0, 0)$,
  stays weakly above the line $y = -t$, and ends on the line $y = 0$.
  The combinatorial class of $k_t$-Dyck paths will be denoted by $\mathcal{S}^{(k)}_t$.
\end{definition}
As in the case of $k$-Dyck paths, a $k_t$-Dyck path with $n$ up-steps has
$kn$ down-steps, and thus its length is $(k+1)n$.

For $0 \leq t \leq k$,
$k_t$-Dyck paths are enumerated by
\emph{generalized Fuss--Catalan numbers} (see for example \OEIS{A001764},
\OEIS{A006013}, \OEIS{A002293}--\OEIS{A002296}, \OEIS{A069271},
\OEIS{A006632}, \OEIS{A118969}--\OEIS{A118971}
in the OEIS~\cite{OEIS:2022}),
as shown in the next proposition.

\begin{proposition}[{\cite[Proposition~2.2.2]{Selkirk:2019:MSc}}]\label{prop:k_t:enum}
  For $0 \leq t \leq k$,
  the generating function for $k_t$-Dyck paths (with the variable $x$
  counting the number of up-steps) is given by
  \begin{equation*}
    S^{(k)}_t(x)  = \sum_{n\geq 0}C_{n, t}^{(k)}x^{n},
  \end{equation*}
  where the generalized Fuss--Catalan numbers $C_{n, t}^{(k)}$ are given by
  \begin{equation}\label{eq:gen-fuss-catalan}
    C_{n, t}^{(k)} = \frac{t+1}{(k+1)n+t+1}\binom{(k+1)n+t+1}{n}.
  \end{equation}
\end{proposition}
For ease of notation we omit the superscript and write
$\mathcal{S}_t$, ${S}_t(x)$, or $C_{n, t}$ if~$k$ is fixed and clear from the context.
Note that for $t=0$ we  obtain the ``classical'' Fuss--Catalan numbers as in \eqref{eq:fuss-catalan}.

This can be proved in a number of ways, as discussed in~\cite{Selkirk:2019:MSc},
the simplest of those discussed being an application of the Cycle Lemma
(see~\cite{Dvoretzky-Motzkin:1947}; or \cite[Section 4.1]{Banderier-Flajolet:2002:lat-path},
\cite{Dershowitz-Zaks:1990:cycle-lemma}, \cite{Lothaire:1997:comb-words}
for lattice path applications). The result can also be found in literature: see, for example,
\cite[Section 1.4, Theorem 3; in particular (1.11)]{Mohanty:1979:lattic} for just the
enumeration, or \cite[Theorem 10.4.5]{Krattenthaler:2015:lattic} where both the
enumeration as well as the generating function representation is proved.
We choose to include the proof using the Cycle Lemma here nonetheless, because it also
sets the stage for some results later in this article.

\begin{proof}[of Proposition~\ref{prop:k_t:enum}]
  Trivially, $k_t$-Dyck paths of length $(k+1)n$ are in bijection with reverse $k_t$-Dyck paths of length $(k+1)n$
  (with step set $\{(1, 1), (1, -k)\}$). Using the Cycle Lemma, we will show that
  reverse $k_t$-Dyck paths of length $(k+1)n$ are enumerated by the generalized
  Fuss--Catalan numbers given in \eqref{eq:gen-fuss-catalan}.
  As outlined in \cite{Dershowitz-Zaks:1990:cycle-lemma}, a sequence $p_1\cdots p_n$
  of boxes and circles is \emph{$k$-dominating} if for every $j$, $1 \leq j \leq n$ the number of boxes in
  $p_1\cdots p_j$ is more than $k$ times the number of circles.

  The Cycle Lemma states that: ``For any sequence $p_1\cdots p_{m+n}$ of $m$ boxes and
  $n$ circles, $m \geq kn$, there exist exactly $m-kn$ (out of $m+n$) cyclic permutations
  $p_j\cdots p_{m+n}p_1\cdots p_{j-1}$, $1\leq j\leq m+n$, that are $k$-dominating.''

  We claim that the number of unique $k$-dominating sequences of $n$ circles and $kn+t+1$ boxes is equal to the number of
  reverse $k_t$-Dyck paths, where a box represents a $(1, 1)$ step, and a circle represents a $(1, -k)$ step in order.
  A $k$-dominating sequence of boxes and circles can be drawn as a path starting at $(0, -(t+1))$ and ending at $((k+1)n+t+1, 0)$,
  and its subpath from $(t+1, 0)$ to $((k+1)n+t+1, 0)$ is a valid reverse $k_t$-Dyck path. We can take this subpath because every 
  $k$-dominating sequence must begin with at least $t+1$ boxes (or $(1, 1)$ steps). Since the sequence is
  $k$-dominating, we have that $k$ times the number of $(1, 1)$ steps is always strictly greater than the number
  of $(1, -k)$ steps, i.\,e.\, $k\cdot \#(1, 1) + 1 \geq \#(1, -k)$, after removing the initial $t+1$ steps of type $(1, 1)$
  the path will possibly have a $(1, -k)$ step at height $k-t$ and we have the condition that the path stays weakly
  above the line $y = -t$. Thus unique $k$-dominating sequences of $n$ circles and $kn+t+1$ boxes and reverse $k_t$-Dyck
  paths are in bijection.

  The number of sequences of $n$ circles and $kn+t+1$ boxes that can be constructed is
  \begin{equation*}
    \binom{(k+1)n+t+1}{n}.
  \end{equation*}
  By applying the Cycle Lemma we see that for each sequence there are exactly $(kn+t+1)-kn = t+1$ starting points for cyclic
  permutations of the sequence which are $k$-dominating. However, these starting points might not
  give unique sequences due to periodicity of the sequence. On the other hand, all non-unique
  $k$-dominating cyclic permutations are periodic with a period that divides the length of the
  sequence. Therefore proportionally there are $(t+1)/((k+1)n+t+1)$ (unique) $k$-dominating cyclic
  permutations per sequence. Hence there are
  \begin{equation*}
    \frac{t+1}{(k+1)n+t+1}\binom{(k+1)n+t+1}{n}
  \end{equation*}
  unique $k$-dominating sequences, which concludes the proof.
\end{proof}

\begin{remark}\label{rem:t>k}
  In a recent preprint by Prodinger \cite{Prodinger:2019:negative-boundary} an
  approach for investigating $k_t$-Dyck paths with $t > k$ using generating functions
  was discussed --- note that the enumeration of these paths is known, and can be obtained
  from, e.g., \cite[Theorem 10.4.7]{Krattenthaler:2015:lattic}.
  Combinatorially, removing the restriction of $0 \leq t \leq k$ results in the paths
  behaving differently with respect to certain decompositions.  As we will see later, the
  approaches presented in our paper cannot be applied directly to the case $t > k$: this
  is a subject of our ongoing research. Therefore throughout this
  work we assume that $0 \leq t\leq k$, unless stated otherwise.
\end{remark}

In this paper, we study the number of down-steps between pairs of up-steps
in $k_t$-Dyck paths.

\begin{definition}[Number of down-steps]\label{def:downsteps}
  Let $n\geq 1$ and $0\leq r\leq n$ be two non-negative
  integers. Let $s_{n, t, r}^{(k)}$ be the
  total number of all
  down-steps between the $r$-th
  and the $(r+1)$-th up-steps in all $k_{t}$-Dyck paths of length $(k+1)n$.
  For $r = 0$ this quantity
  enumerates all down-steps before the
  first up-step, and for $r = n$  all down-steps after
  the last up-step.
\end{definition}
Whenever $k$ is clear from the given context, we write
$s_{n, t, r}$ instead of $s_{n, t,r}^{(k)}$, to improve readability.
Figure~\ref{fig:3-dyck-size-2} illustrates all $3_{1}$-Dyck paths with
$n=2$ up-steps. By counting the down-steps of the paths in the figure
we can see that $s_{2,1,0}^{(3)} = 4$, $s_{2,1,1}^{(3)} = 16$, and $s_{2,1,2}^{(3)} = 34$.

\begin{figure}[ht]
  \begin{minipage}[ht]{0.19\linewidth}\centering
    \begin{tikzpicture}[scale=0.3]
      \draw[help lines] (0,-1) grid (8,6);
      \draw[thick, ->] (-0.25, 0) -- (8.5, 0);
      \draw[thick, ->] (0, -1.25) -- (0, 6.5);
      \drawlatticepath{-1, 3, -1, -1, -1, 3, -1, -1}
    \end{tikzpicture}
    $(1, 3, 2)$
  \end{minipage}
  \begin{minipage}[ht]{0.19\linewidth}\centering
    \begin{tikzpicture}[scale=0.3]
      \draw[help lines] (0,-1) grid (8,6);
      \draw[thick, ->] (-0.25, 0) -- (8.5, 0);
      \draw[thick, ->] (0, -1.25) -- (0, 6.5);it-status
      \drawlatticepath{-1, 3, -1, -1, 3, -1, -1, -1}
    \end{tikzpicture}
    $(1, 2, 3)$
  \end{minipage}
  \begin{minipage}[ht]{0.19\linewidth}\centering
    \begin{tikzpicture}[scale=0.3]
      \draw[help lines] (0,-1) grid (8,6);
      \draw[thick, ->] (-0.25, 0) -- (8.5, 0);
      \draw[thick, ->] (0, -1.25) -- (0, 6.5);
      \drawlatticepath{-1, 3, -1, 3, -1, -1, -1, -1}
    \end{tikzpicture}
    $(1, 1, 4)$
  \end{minipage}
  \begin{minipage}[ht]{0.19\linewidth}\centering
    \begin{tikzpicture}[scale=0.3]
      \draw[help lines] (0,-1) grid (8,6);
      \draw[thick, ->] (-0.25, 0) -- (8.5, 0);
      \draw[thick, ->] (0, -1.25) -- (0, 6.5);
      \drawlatticepath{-1, 3, 3, -1, -1, -1, -1, -1}
    \end{tikzpicture}
    $(1, 0, 5)$
  \end{minipage}
  \begin{minipage}[ht]{0.19\linewidth}\centering
    \begin{tikzpicture}[scale=0.3]
      \draw[help lines] (0,-1) grid (8,6);
      \draw[thick, ->] (-0.25, 0) -- (8.5, 0);
      \draw[thick, ->] (0, -1.25) -- (0, 6.5);
      \drawlatticepath{3, -1, -1, -1, -1, 3, -1, -1}
    \end{tikzpicture}
    $(0, 4, 2)$
  \end{minipage}

  \begin{minipage}[ht]{0.19\linewidth}\centering
    \begin{tikzpicture}[scale=0.3]
      \draw[help lines] (0,-1) grid (8,6);
      \draw[thick, ->] (-0.25, 0) -- (8.5, 0);
      \draw[thick, ->] (0, -1.25) -- (0, 6.5);
      \drawlatticepath{3, -1, -1, -1, 3, -1, -1, -1}
    \end{tikzpicture}
    $(0, 3, 3)$
  \end{minipage}
  \begin{minipage}[ht]{0.19\linewidth}\centering
    \begin{tikzpicture}[scale=0.3]
      \draw[help lines] (0,-1) grid (8,6);
      \draw[thick, ->] (-0.25, 0) -- (8.5, 0);
      \draw[thick, ->] (0, -1.25) -- (0, 6.5);
      \drawlatticepath{3, -1, -1, 3, -1, -1, -1, -1}
    \end{tikzpicture}
    $(0, 2, 4)$
  \end{minipage}
  \begin{minipage}[ht]{0.19\linewidth}\centering
    \begin{tikzpicture}[scale=0.3]
      \draw[help lines] (0,-1) grid (8,6);
      \draw[thick, ->] (-0.25, 0) -- (8.5, 0);
      \draw[thick, ->] (0, -1.25) -- (0, 6.5);
      \drawlatticepath{3, -1, 3, -1, -1, -1, -1, -1}
    \end{tikzpicture}
    $(0, 1, 5)$
  \end{minipage}
  \begin{minipage}[ht]{0.19\linewidth}\centering
    \begin{tikzpicture}[scale=0.3]
      \draw[help lines] (0,-1) grid (8,6);
      \draw[thick, ->] (-0.25, 0) -- (8.5, 0);
      \draw[thick, ->] (0, -1.25) -- (0, 6.5);
      \drawlatticepath{3, 3, -1, -1, -1, -1, -1, -1}
    \end{tikzpicture}
    $(0, 0, 6)$
  \end{minipage}
  \caption{All $3_1$-Dyck paths with two up-steps. The triple below
    each path indicates its contribution to $s_{2,1,0}^{(3)}$, $s_{2,1,1}^{(3)}$, $s_{2,1,2}^{(3)}$ --- that is,
    the number of down-steps before the first up-step,
    between the first and second up-step, and after the second
    up-step, respectively.}
  \label{fig:3-dyck-size-2}
\end{figure}
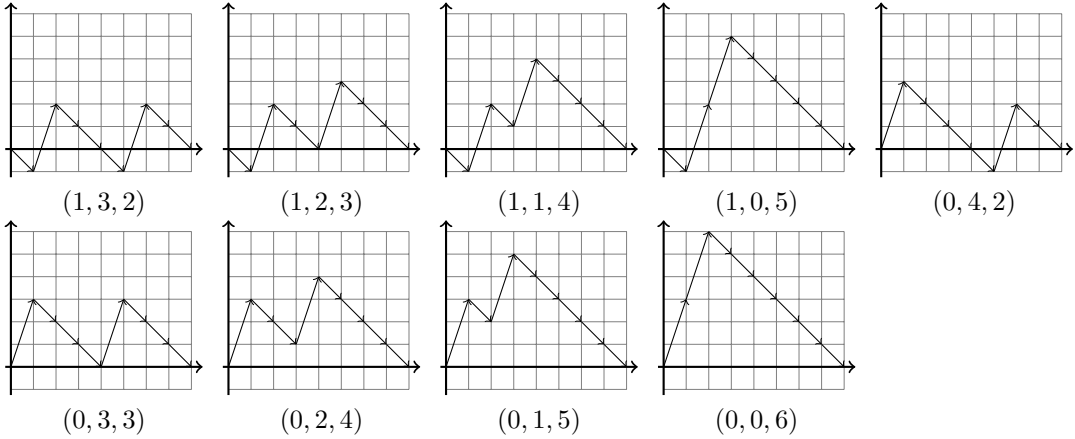

The motivation for studying down-steps in
$k_t$-Dyck paths comes from trying to better understand and enumerate
generalized Kreweras walks.  Walks in the quarter plane have been and
remain an active area of research, with many unresolved questions.
Kreweras walks \cite{Kreweras:1965:OGwalks} are walks in the
quarter plane, $\mathbb{Z}_{\geq 0}\times\mathbb{Z}_{\geq 0}$,
 that consist of steps
 $\{(1, 1), (-1, 0), (0, -1)\}$, and start and end at $(0,0)$.

We consider a generalized Kreweras model with step set
$\{(a, b), (-1, 0), (0, -1)\}$, $a, b \in
\mathbb{N}$. These \emph{$(a,b)$-Kreweras}
walks can be decomposed into pairs consisting of an $a$-Dyck path and
a $b$-Dyck path using an observation in \cite[Proposition
21]{Bostan-BousquetMelou-Melczer:2018:quarter-plane}. Given an
arbitrary $(a,b)$-Kreweras walk, we sequentially consider the
non-zero $x$-coordinates of steps in the walk, which will be either
$a$ or $-1$, and thus form a sequence which represents an $a$-Dyck
path. Similarly, non-zero $y$-coordinates are either $b$ or $-1$ and
thus form a $b$-Dyck path.

However, because not all the conditions of \cite[Proposition 21]{Bostan-BousquetMelou-Melczer:2018:quarter-plane} are
met, this decomposition is not unique --- two different walks can
result in the same pair of $a$-Dyck and $b$-Dyck path. For example,
consider the decomposition in Figure~\ref{fig:diff-walk-same-dec}.
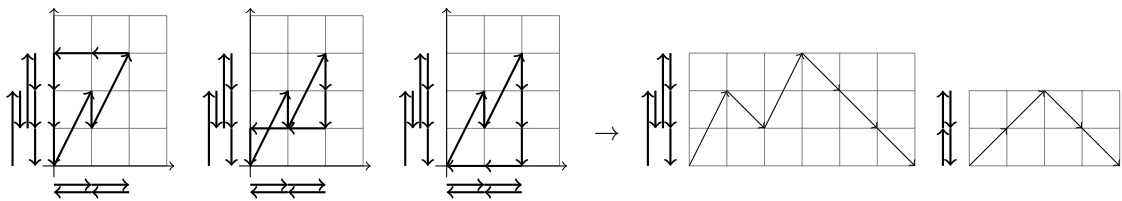
\begin{figure}[ht]
  \begin{tikzpicture}[scale = 0.5]
    \draw[->, thick] (-1.1, 0) to (-1.1, 2);
    \draw[->, thick] (-0.9, 2) to (-0.9, 1);
    \draw[->, thick] (-0.7, 1) to (-0.7, 3);
    \draw[->, thick] (-0.5, 3) to (-0.5, 2);
    \draw[->, thick] (-0.5, 2) to (-0.5, 1);
    \draw[->, thick] (-0.5, 1) to (-0.5, 0);

    \draw[->, thick] (0, -0.5) to (1, -0.5);
    \draw[->, thick] (1, -0.5) to (2, -0.5);
    \draw[->, thick] (2, -0.7) to (1, -0.7);
    \draw[->, thick] (1, -0.7) to (0, -0.7);

    \draw[help lines] (0, 0) grid (3, 4);
    \draw[->] (-0.3, 0) to (3.2, 0);
    \draw[->] (0, -0.3) to (0, 4.2);

    \coordinate (aux) at (0, 0);
    \foreach \i in {(1, 2), (0, -1), (1, 2), (-1, 0), (-1, 0), (0, -1), (0, -1), (0, -1)}
    \draw[->, thick] (aux)--++\i coordinate (aux);
  \end{tikzpicture} \quad
  \begin{tikzpicture}[scale = 0.5]
    \draw[->, thick] (-1.1, 0) to (-1.1, 2);
    \draw[->, thick] (-0.9, 2) to (-0.9, 1);
    \draw[->, thick] (-0.7, 1) to (-0.7, 3);
    \draw[->, thick] (-0.5, 3) to (-0.5, 2);
    \draw[->, thick] (-0.5, 2) to (-0.5, 1);
    \draw[->, thick] (-0.5, 1) to (-0.5, 0);

    \draw[->, thick] (0, -0.5) to (1, -0.5);
    \draw[->, thick] (1, -0.5) to (2, -0.5);
    \draw[->, thick] (2, -0.7) to (1, -0.7);
    \draw[->, thick] (1, -0.7) to (0, -0.7);

    \draw[help lines] (0, 0) grid (3, 4);
    \draw[->] (-0.3, 0) to (3.2, 0);
    \draw[->] (0, -0.3) to (0, 4.2);

    \coordinate (aux) at (0, 0);
    \foreach \i in {(1, 2), (0, -1), (1, 2), (0, -1), (0, -1), (-1, 0), (-1, 0), (0, -1)}
    \draw[->, thick] (aux)--++\i coordinate (aux);

  \end{tikzpicture} \quad
  \begin{tikzpicture}[scale = 0.5]
    \draw[->, thick] (-1.1, 0) to (-1.1, 2);
    \draw[->, thick] (-0.9, 2) to (-0.9, 1);
    \draw[->, thick] (-0.7, 1) to (-0.7, 3);
    \draw[->, thick] (-0.5, 3) to (-0.5, 2);
    \draw[->, thick] (-0.5, 2) to (-0.5, 1);
    \draw[->, thick] (-0.5, 1) to (-0.5, 0);

    \draw[->, thick] (0, -0.5) to (1, -0.5);
    \draw[->, thick] (1, -0.5) to (2, -0.5);
    \draw[->, thick] (2, -0.7) to (1, -0.7);
    \draw[->, thick] (1, -0.7) to (0, -0.7);

    \draw[help lines] (0, 0) grid (3, 4);
    \draw[->] (-0.3, 0) to (3.2, 0);
    \draw[->] (0, -0.3) to (0, 4.2);

    \coordinate (aux) at (0, 0);
    \foreach \i in {(1, 2), (0, -1), (1, 2), (0, -1), (0, -1), (0, -1), (-1, 0), (-1, 0)}
    \draw[->, thick] (aux)--++\i coordinate (aux);

  \end{tikzpicture}\quad \raisebox{2em}{$\rightarrow$}\quad
  \raisebox{1em}{\begin{tikzpicture}[scale = 0.5]
    \draw[->, thick] (-1.1, 0) to (-1.1, 2);
    \draw[->, thick] (-0.9, 2) to (-0.9, 1);
    \draw[->, thick] (-0.7, 1) to (-0.7, 3);
    \draw[->, thick] (-0.5, 3) to (-0.5, 2);
    \draw[->, thick] (-0.5, 2) to (-0.5, 1);
    \draw[->, thick] (-0.5, 1) to (-0.5, 0);

    \draw[help lines] (0, 0) grid (6, 3);
    \drawlatticepath{2, -1, 2, -1, -1, -1}
  \end{tikzpicture}}\quad
  \raisebox{1em}{\begin{tikzpicture}[scale = 0.5]
    \draw[->, thick] (-0.7, 0) to (-0.7, 1);
    \draw[->, thick] (-0.7, 1) to (-0.7, 2);
    \draw[->, thick] (-0.5, 2) to (-0.5, 1);
    \draw[->, thick] (-0.5, 1) to (-0.5, 0);

    \draw[help lines] (0, 0) grid (4, 2);
    \drawlatticepath{1, 1, -1, -1}
  \end{tikzpicture}}
  \caption{Three different $(1,2)$-Kreweras walks whose $x$- and $y$-projections yield the 
  same pair of paths. In total there are ten $(1, 2)$-Kreweras walks yielding this pair of paths.}
  \label{fig:diff-walk-same-dec}
\end{figure}

This is because there is no way of encoding the relative positions of
$(-1, 0)$ and $(0, -1)$ steps between pairs of $(a, b)$ steps, but in the $a$-Dyck and
$b$-Dyck path the corresponding steps are down-steps. Since $(a, b)$ steps result
in an up-step in both generalized Dyck paths, the number of down-steps
between pairs of up-steps in each path allows us to enumerate the
number of combinations of $(-1, 0)$ and $(0, -1)$ steps that give the
same $a$-Dyck and $b$-Dyck path, thus providing a basis for the
enumeration of generalized Kreweras walks. We hope that a better
understanding of the down-step statistics in $k_{t}$-Dyck paths will
assist in the enumeration of this family of generalized
Kreweras walks in the quarter plane.

\section{Structure of the paper and summary of results}\label{sec:results}

There are two different approaches used for investigating the number
of down-steps. Firstly, in Section~\ref{sec:downsteps} we bijectively
prove a recursion which gives an explicit sum formula for the total
number of down-steps between pairs of up-steps. The second approach in
Section~\ref{sec:generating-functions} makes use of the symbolic
decomposition of $k_t$-Dyck paths and uses bivariate generating
functions to count the number of down-steps between pairs of
up-steps. Combining results from both approaches leads to several
interesting identities and observations. See Table~\ref{tab:summary}
for a concise summary including references to related sequences in the
OEIS~\cite{OEIS:2022}.

Section~\ref{sec:applications} outlines connections of both $k_t$-Dyck paths and their down-step statistics to another
combinatorial object used in coding theory --- a family of perforation patterns of punctured convolutional codes --- which
were originally studied by B\'{e}gin \cite{Begin:1992:enumeration-punctured-codes}. These are binary matrices with a
prescribed number of 1's and 0's, and we show that they are bijective to $k_t$-Dyck paths for $0 \leq t < k$ (Theorem~\ref{thm:matrices_paths}).
As a consequence of this new link between binary matrices and lattice paths, a new interpretation of Catalan numbers is obtained; see Corollary~\ref{cor:new-catalan} and Figure~\ref{fig:dyck-cpc}.

The generating function for the number of down-steps between the $r$-th and $(r+1)$-th up-steps in $k_t$-Dyck 
paths can be found in Theorem~\ref{thm:gen-func-down-steps}. 
The asymptotic behavior related to the results in the table below can be found in a
concise format in Section~\ref{sec:asy}. The results are derived from the 
explicit formulas determined in Sections~\ref{sec:downsteps} and~\ref{sec:generating-functions}.

\begin{table}[htbp]
  \caption{Summary of explicit formulas for $s_{n,t,r}^{(k)}$, the
    total number of down-steps between the $r$-th and $(r+1)$-th
    up-steps in $k_{t}$-Dyck paths of length $(k+1)n$, where $k$,
    $n\geq 1$, $0\leq t\leq k$ and $r$ are integers. Restrictions on
    $r$ are listed in the table.}\label{tab:summary}
  \def\arraystretch{1.25}
  \begin{tabular}{|p{\textwidth}|}
    \hline
    \textboxa{Theorem~\ref{thm:downstep-recurrence} (1) and Theorem~\ref{thm:r=0}\hfill}\textboxb{\hfil $r = 0$\hfil}\textboxc{\hfill OEIS \OEIS{A001764}, \OEIS{A002293}, \OEIS{A002294}, \OEIS{A334785}--\OEIS{A334787}}\\
    \hline
    \begin{equation*}
      s_{n, t, 0}^{(k)} = \sum_{j=0}^{t-1} \frac{j+1}{(k+1)n+j+1}\binom{(k+1)n+j+1}{n} = \frac{k}{n+1}\binom{(k+1)n}{n} - \frac{k-t}{n+1}\binom{(k+1)n+t}{n}
    \end{equation*}\\
    \hline
    \textboxa{Corollary~\ref{cor:downsteps-explicit} and Corollary~\ref{cor:downsteps-explicit-gf}\hfill}\textboxb{\hfil $0 \leq r < n$\hfil}\textboxc{\hfill OEIS \OEIS{A007226}, \OEIS{A007228}, \OEIS{A124724}, \OEIS{A334640}--\OEIS{A334651}}\\
    \hline
    {\begin{align*}
      s_{n,t,r}^{(k)}
      & = \frac{k}{n+1}\binom{(k+1)n}{n} - \frac{k-t}{n+1}\binom{(k+1)n+t}{n}\\
      & \hspace{3cm} + \sum_{j=1}^{r} \frac{t+1}{(k+1)j + t + 1}\frac{k}{n-j+1} \binom{(k+1)j + t +1}{j} \binom{(k+1)(n-j)}{n-j}
     \end{align*}}\\
    \hline
    \textboxa{Corollary~\ref{cor:r-last-downsteps}\hfill}\textboxb{\hfil $1 \leq r \leq n$\hfil}\textboxc{\hfill OEIS \OEIS{A334976}--\OEIS{A334980}}\\
    \hline
    {\begin{align*}
       &s_{n,t,n-r}^{(k)}\\
       & \quad = \frac{t+1}{(k+1)(n+1)+t+1}\binom{(k+1)(n+1)+t+1}{n+1} - \frac{(k+1)(t+1)}{(k+1)n+t+1}\binom{(k+1)n+t+1}{n}  \\
       & \qquad - \sum_{j=2}^{r} \frac{1}{j-1}\binom{(j-1)(k+1)}{j}\,\frac{t+1}{(k+1)(n-j+1) + t + 1}\binom{(k+1)(n-j+1) + t + 1}{n-j+1}
     \end{align*}}\\
    \hline
    \textboxa{Proposition~\ref{prop:last-downsteps} and Proposition~\ref{thm:down-steps-end}\hfill}\textboxb{\hfil $r = n$\hfil}\textboxc{\hfill OEIS \OEIS{A030983}, \OEIS{A334609}--\OEIS{A334612}, \OEIS{A334680}, \OEIS{A334682}, \OEIS{A334719}}\\
    \hline
    \begin{equation*}
      s_{n,t,n}^{(k)} =
      \frac{t+1}{(k+1)(n+1)+t+1}\binom{(k+1)(n+1)+t+1}{n+1} -
      \frac{(t+1)^2}{(k+1)n+t+1}\binom{(k+1)n+t+1}{n}
    \end{equation*}\\
    \hline
  \end{tabular}
\end{table}

\section{Down-step statistics: A bijective approach}\label{sec:downsteps}

The main result of this section is the recurrence relation for
$s_{n,t,r}$ given in \eqref{eq:downstep-recurrence} together with relevant
initial values given in~\eqref{eq:downstep-recurrence:ell0} and~\eqref{eq:downstep-recurrence:elevated}.  This
recurrence fully explains the combinatorial structure behind the
quantities given by $s_{n,t,r}$.

\begin{theorem}\label{thm:downstep-recurrence}
  Let $C_{n, j}$ be the generalized Fuss--Catalan numbers from \eqref{eq:gen-fuss-catalan}, $k\geq 1$ and $0 \leq t \leq k$ be fixed integers, and $n\geq 1$ be a
  positive integer. Then the following statements hold:
  \begin{enumerate}
  \item\label{itm:thm:downstep-recurrence:r0} The total number of down-steps before the first
    up-step in all $k_{t}$-Dyck paths of length $(k+1)n$ is given by
    \begin{equation}\label{eq:downstep-recurrence:ell0}
      s_{n, t, 0} = \sum_{j=0}^{t-1} C_{n,j}.
    \end{equation}
  \item\label{itm:thm:downstep-recurrence:elevated} The total number of down-steps between the first and second up-steps in all
    $k_{0}$-Dyck paths is given by
    \begin{equation}\label{eq:downstep-recurrence:elevated}
      s_{n, 0, 1} =  \sum_{j=0}^{k-1} C_{n-1, j} = \frac{k}{n} \binom{(k+1)(n-1)}{n-1}.
    \end{equation}
  \item\label{itm:thm:downstep-recurrence:rec} For all positive integers $r$ with $1\leq r\leq n$ and $(t, r) \neq (0, 1)$, the following recurrence
    relation\footnote{We use the Iverson bracket
      popularized in \cite[Chapter 2]{Graham-Knuth-Patashnik:1994}
      where $\iverson{A} = 1$ if $A$ is true, and $\iverson{A} = 0$
      if $A$ is false.} holds:
    \begin{equation}\label{eq:downstep-recurrence}
      s_{n,t,r} = s_{n,t,r-1} + C_{r, t} (s_{n-r+1, 0, 1} -
      t\iverson{r = n}).
    \end{equation}
  \end{enumerate}
\end{theorem}
We prove all of the statements of this theorem by means of combinatorial arguments such as
bijections and double counting. Additionally, we verify the recurrence relation for small values of
$n$, $k$, and $t$ in an associated SageMath~\cite{SageMath:2020:9.0} worksheet\footnote{The interactive worksheets
  related to this article can be found at
  \url{https://gitlab.aau.at/behackl/kt-dyck-downstep-code}.}.

Let $r$ be an integer where $0\leq r\leq n$. We introduce \emph{$r$-marked
$k_{t}$-Dyck paths} as $k_{t}$-Dyck paths where exactly one of the
down-steps between the $r$-th and the $(r+1)$-th up-steps is marked.
By construction, the number of $r$-marked $k_{t}$-Dyck paths of length
$(k+1)n$ is equal to $s_{n,t,r}$.

Also, notice that for $t=0$ we have
$s_{n,0,0}=0$, and~\eqref{eq:downstep-recurrence} does not produce
$s_{n,0,1}$. Therefore for $t=0$ initial values $s_{n,0,1}$ are
explicitly given in~\eqref{eq:downstep-recurrence:elevated}.

  \begin{proof}[of Theorem~\ref{thm:downstep-recurrence}~(\ref{itm:thm:downstep-recurrence:r0})]
    For $t=0$ we have $s_{n,0,0}=0$ since $k_{0}$-Dyck paths cannot start with a down-step.
    Accordingly, the sum in the right-hand side of \eqref{eq:downstep-recurrence:ell0} is empty.

    For $t>0$, we prove~\eqref{eq:downstep-recurrence:ell0} via a cyclic shift argument.
    Consider a $0$-marked $k_t$-Dyck path
    where the $j$-th down-step at the beginning of the path is marked.
    We cut out the first $j$ down-steps from the beginning of the path and attach them to the end of the path, removing the marking,
    and shift the path by $j$ units upwards (to make it start and end at the $x$-axis). In this way
    the $k_t$-Dyck path is transformed to an unmarked $k_{t-j}$-Dyck
    path. Conversely, given a $k_{t-j}$-Dyck path, the corresponding $0$-marked $k_t$-Dyck path
    can be constructed by marking the last step (which is always a down-step since $j > 0$ and $t-j<t\le k$), removing the
    final $j$ down-steps, adding them back to the beginning of the path,
    and relocating the newly constructed path such that it starts at the origin.

    This proves that for $t>0$, $0$-marked $k_{t}$-Dyck paths are in bijection with
    the tuples in $\cup_{1\leq j\leq t}\{j\} \times \mathcal{S}_{t-j}$,
    where $j$ is required to make the union disjoint and to remember the
    number of steps moved to the end. Since
    $\mathcal{S}_{j}$ is enumerated by the generalized Fuss--Catalan number $C_{n,j}$
    (see Proposition~\ref{prop:k_t:enum}), this
    proves~\eqref{eq:downstep-recurrence:ell0}.
  \end{proof}

  Within the
  following proof, we show that the numbers $s_{n,0,1}$ also enumerate an associated
  class of lattice walks which occurs in the proof of the
  recurrence in~\eqref{eq:downstep-recurrence} (this also explains why
  these special values occur prominently in the recurrence relation).
  This class is introduced in the following definition.
  \begin{definition}[Elevated $k_0$-Dyck path]\label{def:elevated}
    An \emph{elevated} $k_0$-Dyck path is a lattice path
    consisting of \emph{up-steps} $(1,k)$ and \emph{down-steps} $(1,-1)$,
    that starts and ends at the line $y=j$ for any integer $0\leq j< k$,
    and stays weakly above the $x$-axis. The number $j$ is called the \emph{elevation} of the path.
  \end{definition}

  The following observation relating $k_{t}$-Dyck paths to elevated
  $k_{0}$-Dyck paths is straightforward; nevertheless we choose to state
  it formally because we will reference it several times
  later on.

  \begin{observation}\label{obs:elevated}
   For each $0 \leq j < k$,
    the vertical shift by $j$ units is a bijection between
    $k_j$-Dyck paths
    and elevated $k_0$-Dyck paths of elevation $j$.
    Therefore, the number of elevated $k_0$-Dyck paths of elevation $j$ is $C_{n,j}$
    (see Proposition~\ref{prop:k_t:enum}).
  \end{observation}

  \begin{proof}[of Theorem~\ref{thm:downstep-recurrence}~(\ref{itm:thm:downstep-recurrence:elevated})]
        The path decomposition illustrated in
    Figure~\ref{fig:d-dyck:falls-decomposition} shows that
    $1$-marked $k_{0}$-Dyck paths with $n$ up-steps are
    in bijection with elevated $k_{0}$-Dyck paths\footnote{
      Observe that as an elevated path with elevation $j$ starts
      at $(0, j)$, the elevation is inherently stored within the
      path itself. This is in contrast to the paths from $\mathcal{S}_j$
      discussed in part~(\ref{itm:thm:downstep-recurrence:r0}) where
      we formally needed to construct a tuple storing the number of shifted
      steps $j$ to make the union over several path classes $\mathcal{S}_{t-j}$
      disjoint.
    } with $n-1$ up-steps: the middle segment starting at the end of the
    marked down-step and reaching until the final return to the same height
    is an elevated $k_0$-Dyck path with $n-1$ up-steps.
    \begin{figure}
      \begin{tikzpicture}[scale=0.55]
        \draw[help lines] (0,0) grid (10,5);
        \draw[thick, -{latex}] (-0.5, 0) -- (10.5,0);
        \draw[thick, -{latex}] (0, -0.5) -- (0, 5.5);
        \draw[->] (0,0) -- (1, 4);
        \draw[->] (1,4) -- (2,3);
        \draw[->, very thick] (2,3) -- (3,2);
        \draw[dashed] (3,2) .. controls (4,0) and (5,5) .. (6,0) .. controls (7,5) .. (8,2);
        \draw[->] (8,2) -- (9,1);
        \draw[->] (9,1) -- (10,0);
        \draw [thick, decoration={
          brace,
          mirror,
          raise=0.2cm
        },
        decorate
        ] (3,0) -- (8,0)
        node [pos=0.5,anchor=north,yshift=-0.3cm] {middle segment};
      \end{tikzpicture}
      \caption{Decomposition of a marked $k_{0}$-Dyck path: The
        middle segment corresponds to an elevated $k_{0}$-Dyck path.}
      \label{fig:d-dyck:falls-decomposition}
    \end{figure}
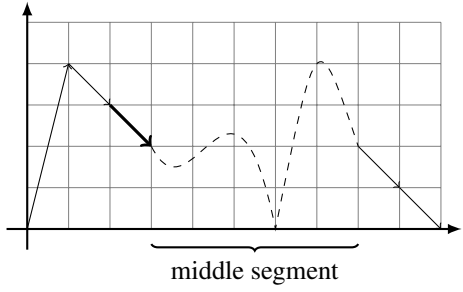
    
    This proves that there are equally many down-steps after the first
    up-step in all $k_0$-Dyck paths with $n$ up-steps as elevated $k_0$-Dyck
    paths with $n-1$ up-steps. By Observation~\ref{obs:elevated} and
    summing over all possible elevations $0\leq j < k$, this number is
    $\sum_{j=0}^{k-1} C_{n, j}$. Moreover, it
    was proved by Gu, Prodinger, and Wagner in~\cite[Section
    3]{Gu-Prodinger-Wagner:2010:k-plane-trees} that
    the number of elevated $k_{0}$-Dyck paths with $n-1$ up-steps and elevation $j$ with $0\le j<k$
     is
     \begin{equation}\label{eq:GPW}
     \frac{k}{n} \binom{(k+1)(n-1)}{n-1},
     \end{equation}
       which is equal to the expression on the right-hand side
    of~\eqref{eq:downstep-recurrence:elevated}.
    This completes the proof.
  \end{proof}

The introduction of elevated $k_0$-Dyck paths together with the bijection
constructed in the previous proof allows us to extend the statement in
Theorem~\ref{thm:downstep-recurrence}~(\ref{itm:thm:downstep-recurrence:elevated}).

\begin{corollary}\label{cor:1-marked-elevated}
  Let $k \geq 1$ be a fixed integer, and let $n\geq 1$ be a positive integer.
  Then $s_{n, 0, 1} = \frac{k}{n} \binom{(k+1)(n-1)}{n-1}$ enumerates
  $1$-marked $k_0$-Dyck paths with $n$ up-steps. Furthermore, these paths
  are in bijection with elevated $k_0$-Dyck paths with $n-1$ up-steps.
\end{corollary}

\begin{remark}\label{rem:bijective-proof}
Notice that our reasoning contains a bijective proof of the nice summation formula
  \begin{equation}\label{eq:sum-of-k_t}
    \sum_{j=0}^{k-1} C_{n, j}
    = \sum_{j=0}^{k-1} \frac{j+1}{(k+1)n + j + 1}\binom{(k+1)n + j + 1}{n}
    = \frac{k}{n+1} \binom{(k+1)n}{n}.
  \end{equation}
  As can be expected from its simple form, this identity is well-known and
  can also be proved using algebraic techniques.
  For some fixed values of $k$ and $n$, these numbers can be found in
  OEIS entries \OEIS{A007226}, \OEIS{A007228}, \OEIS{A124724}.

  Moreover, in the formula $s_{n, t, 0} = \sum_{j=0}^{t-1} C_{n,j}$ we
  have a partial sum of $s_{n+1, 0, 1} = \sum_{j=0}^{k-1} C_{n, j}$. 
  In the special case $t=k$ this yields
  $s_{n,k,0}=s_{n+1, 0, 1}$. In Proposition~\ref{prop:snkr} we return
  to and generalize this equality, and provide a bijective proof.
\end{remark}

Now that we have proved the initial values of $s_{n, t, r}$, we
proceed with the proof of the recurrence.

  \begin{proof}[of Theorem~\ref{thm:downstep-recurrence}~(\ref{itm:thm:downstep-recurrence:rec})]
    As before, we represent the left-hand side of~\eqref{eq:downstep-recurrence},
    $s_{n, t, r}$, by the number of $r$-marked $k_{t}$-Dyck paths of
    length $(k+1)n$ (consisting of $n$ up-steps).  In
    Claims~\ref{clm:recurrence:shift-peaks}
    and~\ref{clm:recurrence:middle-segment} we prove that the
    right-hand side of~\eqref{eq:downstep-recurrence}
    is also equal to $s_{n, t, r}$, with each summand counting
    paths whose marked step ends in a specified height interval.
  \begin{claim}\label{clm:recurrence:shift-peaks}
    The term $s_{n,t,r-1}$ enumerates $r$-marked $k_{t}$-Dyck paths
    where the marked step ends on height $h$ where $h \geq k-t$.
  \end{claim}
  \begin{proof}
    We construct a bijection that maps $(r-1)$-marked $k_{t}$-Dyck
    paths to $r$-marked $k_{t}$-Dyck paths where the marked down-step
    ends at a height of at least $k-t$. We do this by shifting the
    $(r-1)$-th up-step so that it becomes the $r$-th one. Effectively,
    this moves a peak\footnote{A peak in a $k_{t}$-Dyck path is
      an up-step followed by a
    down-step.} in the path to another position, thus we
    will refer to the procedure as the ``peak shift'' bijection.

    Consider any $(r-1)$-marked $k_{t}$-Dyck path. Let $h$
    be the height on which the last down-step between the $(r-1)$-th
    and the $r$-th up-steps ends, then traverse the path from the
    end of this last downstep towards the beginning. The traversal ends if
    \begin{itemize}
    \item we reach a height of $h$ or below again (as in
      Figure~\ref{fig:peak-shift:normal}), or
    \item we reach the
    beginning of the path without doing so (as in
      Figure~\ref{fig:peak-shift:special}).
    \end{itemize}

    In the second case, prepend the $r$-th up-step to
    the beginning of the traversed segment. This ensures that the
    traversed segment starts at a lower or equal
    height compared to its end.

    The path is then transformed by removing the traversed segment
    and inserting it at the 
    rightmost position located above height $k-t$ after the former $r$-th up-step and
    before the $(r+1)$-th up-step (should it exist).
    This results in an $r$-marked path, with the marked step
    guaranteed to end above or at height $k-t$.

    \begin{figure}[ht]
      \begin{subfigure}[t]{0.48\linewidth}
      \begin{tikzpicture}[scale=0.39]
        \draw[help lines] (0,-1) grid (9,4);
        \draw[thick, ->] (-0.25, 0) -- (9.5, 0);
        \draw[thick, ->] (0, -1.25) -- (0, 4.5);
        \drawlatticepath{-1,2,-1,2,-1,-1,-1,2,-1}
        \draw[ultra thick, black, ->] (6,0) -- (7,-1);
        \draw[decoration={brace,mirror,raise=3pt},decorate] (1,-1) --
        node[below=4pt] {\scriptsize traversed segment} (7,-1);
      \end{tikzpicture}
      \begin{tikzpicture}[scale=0.39]
        \draw[help lines] (0,-1) grid (9,4);
        \draw[thick, ->] (-0.25, 0) -- (9.5, 0);
        \draw[thick, ->] (0, -1.25) -- (0, 4.5);
        \drawlatticepath{-1, 2, 2, -1, 2, -1, -1, -1, -1}
        \draw[ultra thick, black, ->] (7,2) -- (8,1);
        \draw[decoration={brace,mirror,raise=3pt},decorate] (2,-1) --
        node[below=4pt] {\scriptsize inserted segment} (8,-1);
      \end{tikzpicture}
      \subcaption{The ending height of the down-step sequence containing the
        marked down-step is reached earlier in the
        path.}\label{fig:peak-shift:normal}
      \end{subfigure}\hfill
      \begin{subfigure}[t]{0.5\linewidth}
        \begin{tikzpicture}[scale=0.39]
        \draw[help lines] (-1,-1) grid (9,3);
        \draw[thick, ->] (-1.25, 0) -- (9.5, 0);
        \draw[thick, ->] (0, -2.25) -- (0, 3.5);
        \drawlatticepath{2,-1,2,-1,-1,-1}
        \draw[dashed, ->] (7, -1) -- (8, 1);
        \draw[->] (8,1) -- (9,0);
        \draw[dashed, ->] (-1, -2) -- (0,0);
        \draw[ultra thick, black, ->] (6,0) -- (7,-1);
        \draw[decoration={brace,mirror,raise=3pt},decorate] (-1,-2) --
        node[below=4pt] {\scriptsize traversed segment} (7,-2);
      \end{tikzpicture}
      \begin{tikzpicture}[scale=0.39]
        \draw[help lines] (0,0) grid (9,5);
        \draw[thick, ->] (-0.25, 0) -- (9.5, 0);
        \draw[thick, ->] (0, -0.25) -- (0, 5.5);
        \drawlatticepath{2, 2, -1, 2, -1, -1, -1, -1, -1}
        \draw[ultra thick, black, ->] (7,2) -- (8,1);
        \draw[decoration={brace,mirror,raise=3pt},decorate] (0,0) --
        node[below=4pt] {\scriptsize inserted segment} (8,0);
      \end{tikzpicture}
      \subcaption{The ending height of the down-step sequence containing
        the marked down-step is not reached earlier in the path; the
        traversed segment extends until the beginning of the
        path. The dashed arrow represents the ``borrowed'' up-step.}\label{fig:peak-shift:special}
      \end{subfigure}
      \caption{Two examples of the ``peak shift'' bijection mapping
        $(r-1)$-marked paths to $r$-marked paths whose marked step
        ends on height $k-t$ or above. The parameters in the
        illustrations correspond to $k = 2$, $t = 1$, and
        $n = r = 3$.}
      \label{fig:peak-shift}
    \end{figure}
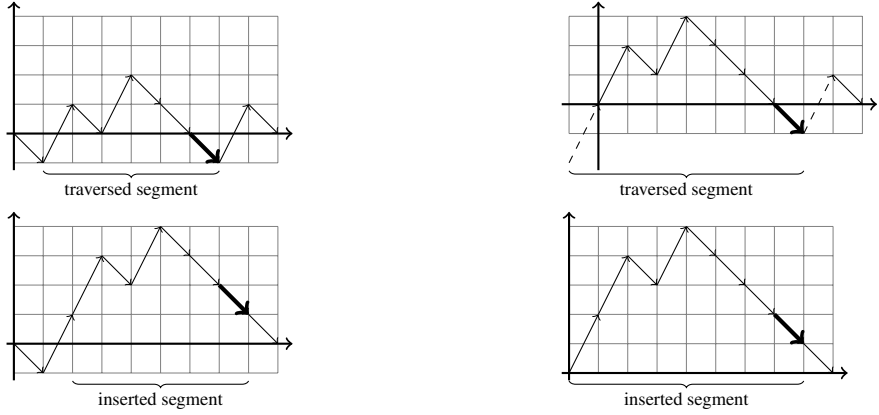

    This transformation can always be reversed: consider an $r$-marked
    $k_t$-Dyck path whose marked down-step ends on a height of at least $k-t$.
    By construction, in the maximal consecutive sequence of down-steps which contains 
    the marked down-step, all steps ending at a height of at least $k - t$ 
    must have been included in the relocated traversed segment; let $h$ be
    this height. Start at the down-step ending at height $h$ and traverse the path towards
    the beginning until you reach height $h$ (or below)
    again.

    If the traversed segment spans until the beginning of the path,
    or if there are no further up-steps before the traversed segment, then
    this instance corresponds to the second case --- otherwise
    it corresponds to the first case. Then, the transformation
    can be reversed by relocating the traversed segment (and 
    for the second case additionally adjusting the initial up-step).

    This establishes that the transformation is indeed a bijection,
    and thus proves the claim.
  \end{proof}

  \begin{claim}\label{clm:recurrence:middle-segment}
    The term $C_{r, t} (s_{n-r+1, 0, 1} - t\iverson{r = n})$
    enumerates $r$-marked $k_{t}$-Dyck paths where the marked step ends
    on height $h$ where $-t \leq h < k - t$.
  \end{claim}
  \begin{proof}
    We use a similar ``middle segment''-decomposition as in the proof
    of Theorem~\ref{thm:downstep-recurrence}~(\ref{itm:thm:downstep-recurrence:elevated}).
    
    Let us first consider the case of $r < n$.
    If the marked step ends above or on the $x$-axis,
    we decompose the path into a ``middle segment'' starting after the
    marked down-step (which means it has a height between $0$ and
    $k-t-1$) and ending with the final return to the same height it started at.
    If the marked step ends below the $x$-axis, the middle
    segment extends until the end of the path. To have the segment end on
    the same height it starts at in both cases, we remove the down-steps from
    the $x$-axis until the start of the middle segment and attach them again
    at the end of the middle segment.
    This middle segment is an elevated $k_{0}$-Dyck path with
    $n - r > 0$ up-steps which has been shifted downwards by $t$ units,
    which is visualized in Figure~\ref{fig:path-decomposition}. By 
    Corollary~\ref{cor:1-marked-elevated}, the number of such paths
    is~$s_{n-r+1, 0, 1}$.
    Now consider the path obtained after removing the middle segment: it
    is precisely a $k_t$-Dyck path with $r$ up-steps, and there are $C_{r,t}$
    such paths. Note that we can ignore the marking because its
    location is already encoded by the elevation of the middle segment.

    For the case of $r = n$ we can use the same decomposition --- however,
    in this case the ``middle segment'' is empty. There are $s_{1, 0, 1} = k$ such paths
    (the empty path with different elevations) in total, but as the
    final sequence of down-steps in the original path ends at height $y=0$,
    the shifted elevated paths with elevation $0$, $1$, \ldots, $t-1$ cannot
    occur --- which is why these $t$ paths have to be excluded. This
    explains the occurrence of the Iverson bracket in~\eqref{eq:downstep-recurrence}.

    \begin{figure}[ht]
      \begin{tikzpicture}[scale=0.45]
        \draw[help lines] (0,-1) grid (16,4);
        \draw[thick, ->] (-0.25, 0) -- (16.5, 0);
        \draw[thick, ->] (0, -1.25) -- (0, 4.5);
        \drawlatticepath{-1,3,-1,3,-1,-1,-1,-1,3,-1,-1,-1,-1,3,-1,-1}
        \draw[ultra thick, black, ->] (6,2) -- (7,1);
        \draw[decoration={brace,mirror,raise=3pt},decorate] (7,-1) --
        node[below=4pt] {\scriptsize middle segment} (15,-1);
      \end{tikzpicture}
      \caption{Illustration of the ``middle segment'' (an elevated
        $3_{0}$-Dyck path of elevation 2) in a $2$-marked
        $3_{1}$-Dyck path. Removing the segment yields a $3_{1}$-Dyck path with 2 up-steps.
        The marked step can be identified uniquely from the elevation of the middle segment,
        which allows us to forget the marking.}
      \label{fig:path-decomposition}
    \end{figure}
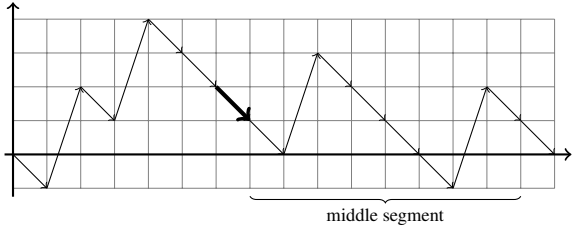

    By following the steps above in reverse order,
    the original path can be reconstructed and so the decomposition is a bijection
    between an $r$-marked $k_{t}$-Dyck path and both a shorter $k_{t}$-Dyck
    path and an elevated $k_{0}$-Dyck path. This proves the claim.
  \end{proof}

  Now, with the statements of Claims~\ref{clm:recurrence:shift-peaks}
  and~\ref{clm:recurrence:middle-segment} we see that the
  right-hand side of~\eqref{eq:downstep-recurrence} enumerates all
  $r$-marked $k_{t}$-Dyck paths --- as does the corresponding left-hand
  side. This proves the recurrence and completes the proof of
  Theorem~\ref{thm:downstep-recurrence}.
\end{proof}

\begin{remark}\label{rem:t>k-explanation}
  As mentioned in Remark~\ref{rem:t>k}, Theorem~\ref{thm:downstep-recurrence} is not valid
  for $t > k$. This is because for $t > k$ the combinatorial structure
  of these paths changes: it is no longer guaranteed that the path is
  on or above the $x$-axis after an up-step. For the same reason, the
  formula from Proposition~\ref{prop:k_t:enum} is invalid for $t > k$.
  In this case, $k_{t}$-Dyck paths are not enumerated by generalized
  Fuss--Catalan numbers, but rather by weighted sums thereof
  (see~\cite{Prodinger:2019:negative-boundary} for details).
\end{remark}

There is more about the case of $t = k$ that is worth mentioning from a combinatorial point of view.
\begin{remark}\label{rem:t=k,r=n}
  Observe that for $t = k$ and $r = n$, the recurrence~\eqref{eq:downstep-recurrence}
  degenerates to $s_{n, k, n} = s_{n, k, n-1}$. Combinatorially, this happens
  because the summand treated
  by Claim~\ref{clm:recurrence:middle-segment} corresponds to the number of down-steps
  before the last return to the axis that are at the same time located below the
  axis --- which is not possible.
  This formula also has a bijective interpretation: observe that if one exchanges
  the down-steps before and after the last up-step in a $k_k$-Dyck
  paths, another valid $k_k$-Dyck path is obtained.
\end{remark}
The cases $t = 0$ and $t = k$ are inherently linked, as the next
proposition shows.

\begin{proposition}\label{prop:snkr}
  Let $n \geq 1$. Then the relation $s_{n+1, 0, r+1} = s_{n, k, r}$
  holds for all integers $r$ with $0\leq r < n$, and for $r = n$ we have
  $s_{n+1, 0, n+1} = s_{n, k, n} + k C_{n,k}$.
\end{proposition}
\begin{proof}
  Let $0\leq r\leq n$. Notice that if we take a $k_{k}$-Dyck path of
  length $(k+1)n$, add a new up-step to the beginning and $k$
  down-steps to the end, the result is a $k_{0}$-Dyck path of length
  $(k+1)(n+1)$.
  Removing the initial up-step and the final $k$ down-steps from a
  $k_{0}$-Dyck path reverses this construction and yields a
  $k_{k}$-Dyck path, thus forming a bijection.

  This bijection transforms the $r$-th up-step in a $k_{k}$-Dyck path
  to the $(r+1)$-th up-step in the resulting $k_{0}$-Dyck
  path, which directly proves $s_{n+1,0,r+1} = s_{n, k , r}$ for
  $0\leq r < n$. For $r = n$, the additional summand $k C_{n,k}$ reflects the
  fact that $k$ additional down-steps are appended to the end of every
  $k_{k}$-Dyck path.
\end{proof}

With the recurrence from Theorem~\ref{thm:downstep-recurrence} we can
find an explicit summation formula for $s_{n,t,r}$.
\begin{corollary}\label{cor:downsteps-explicit}
  For positive integers $k$ and $n$, and integers $r$ and $t$ with
  $0\leq r< n$ and $0\leq t \leq k$, the
  number of down-steps between the $r$-th and $(r+1)$-th up-steps
  in $k_{t}$-Dyck paths of length $(k+1)n$ can be expressed explicitly as
  \begin{equation}\label{eq:downsteps-explicit}
    s_{n,t,r}
    = \sum_{j=0}^{t-1} C_{n, j} + \sum_{j=1}^{r} C_{j,t} s_{n-j+1,0,1}.
\end{equation}
\end{corollary}
The above expression can be found in terms of binomial coefficients in Table~\ref{tab:summary}. 
These sequences for small cases of $k$ and $t$ have been added to the
OEIS, see entries \OEIS{A007226}, \OEIS{A007228}, \OEIS{A124724},
\OEIS{A334640}--\OEIS{A334651}.

Observe that for easier reading, we restricted ourselves to
the case of $r < n$. A similar explicit formula can also
be given for $r = n$. However, we can find a better formula
for this boundary case via a different combinatorial argument.
\begin{proposition}\label{prop:last-downsteps}
  The number of down-steps following the last up-step in $k_{t}$-Dyck paths
  satisfies
  \begin{align}\label{eq:last-downsteps}
    s_{n, t, n}
    &= C_{n+1,t} - (t+1) C_{n,t}
  \end{align}
  for $n\geq 1$, with initial value $s_{0,t,0} = 1$.
\end{proposition}
\begin{proof}
  Consider the rearranged form
  \begin{equation}\label{eq:last-downsteps:combinatorial}
    C_{n+1, t} = (t+1) C_{n,t} + s_{n, t, n}.
  \end{equation}

  The left-hand side of~\eqref{eq:last-downsteps:combinatorial}
  enumerates $k_t$-Dyck paths with $n+1$ up-steps. We claim that the right-hand
  side does the same, where the summands correspond to different ways of
  expanding a $k_t$-Dyck path with $n$ up-steps by one additional up-step.

  In particular, the quantity $s_{n,t,n}$ enumerates $n$-marked $k_t$-Dyck paths with
  $n$ up-steps. Such a path can be extended by replacing the marked down-step by
  an up-step, and adding $k+1$ down-steps after it.

  The longer paths that cannot be constructed in this way are those where
  the $(n+1)$-th up-step starts on height $h$ with $-t\leq h \leq 0$. However,
  we are able to construct them by starting with an arbitrary $k_t$-Dyck path
  with $n$ up-steps (enumerated by $C_{n,t}$), appending between $0$ and $t$
  down-steps to the end of the path (which gives $t+1$ possibilities) followed
  by the $(n+1)$-th up-step and sufficiently many down-steps to make it a
  $k_t$-Dyck path again.

  This proves~\eqref{eq:last-downsteps:combinatorial}, and
  thus~\eqref{eq:last-downsteps}.
\end{proof}

Finally, we combine Proposition~\ref{prop:last-downsteps} with
Theorem~\ref{thm:downstep-recurrence}~(\ref{itm:thm:downstep-recurrence:elevated},
\ref{itm:thm:downstep-recurrence:rec}) to obtain another formula for
$s_{n,t,n-r}$.

\begin{corollary}\label{cor:r-last-downsteps}
  For $1 \leq r \leq n$ we have
  \begin{align}\label{eq:r-last-downsteps}
    s_{n, t, n-r}
    &= C_{n+1,t} - (k+1) C_{n,t} - \sum_{j=2}^{r} \frac{1}{j-1}\binom{(j-1)(k+1)}{j} C_{n-j+1,t}.
  \end{align}
\end{corollary}
\begin{proof}
For $r=1$, we can combine the expressions in~\eqref{eq:downstep-recurrence:elevated},
\eqref{eq:downstep-recurrence}, and~\eqref{eq:last-downsteps} to obtain
\begin{equation*}
s_{n,t,n-1} =
s_{n,t,n} - (s_{1,0,1}-t)C_{n,t} =
C_{n+1,t} - (t+1) C_{n,t} - (k-t) C_{n,t} =
C_{n+1,t} - (k+1) C_{n,t}.
\end{equation*}
Similarly, when $r>1$ we can use induction to obtain
\begin{align*}
s_{n,t,n-r}
& = s_{n,t,n-(r-1)} - s_{r,0,1} C_{n-(r-1),t}
= C_{n+1,t} - (k+1) C_{n,t} - \\
& \quad - \sum_{j=2}^{r-1} \frac{1}{j-1}\binom{(j-1)(k+1)}{j} C_{n-j+1,t}
- \frac{1}{r-1}\binom{(r-1)(k+1)}{r} C_{n-(r-1),t},
\end{align*}
where we use the identity $\frac{1}{b+1} \binom{a}{b} = \frac{1}{a-b} \binom{a}{b+1}$
for positive integers $a > b$.
This yields~\eqref{eq:r-last-downsteps}.
\end{proof}

Asymptotic investigations into the quantities discussed in this
section will be conducted in Section~\ref{sec:asy}.

\section{A connection to coding theory: Binary matrices and punctured codes}\label{sec:applications}
In this section we point out a surprising link between the down-step
parameter $s_{n, t, r}$, and punctured convolutional codes which are
an important concept in coding theory.

\subsection{Perforation patterns of punctured convolutional codes}\label{sec:punctured-codes}

Error correcting codes are a tool for adding redundancy to (binary) data before transmitting
it via some channel that may introduce transmission errors. There are
two large families of codes: block codes, which map a fixed number of input
bits to their respective encoding; and convolutional codes, which process an
input stream by producing an output for overlapping input blocks (these
``windows'' move one bit at a time). See~\cite{Johannesson-Zigangirov:2015:convolutional-coding}
for more basic details on error correcting codes.

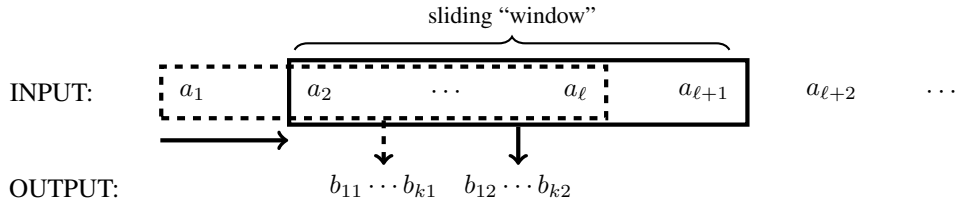
\begin{figure}[ht]
  \begin{tikzpicture}[scale = 0.85]
    \node at (-2.2, 0) {INPUT:};
    \node at (-2, -1.5) {OUTPUT:};

    \node (1) at (0, 0) {$a_1$};
    \node (2) at (2, 0) {$a_2$};
    \node (3) at (4, 0) {$\cdots$};
    \node (4) at (6, 0) {$a_{\ell}$};
    \node (5) at (8, 0) {$a_{\ell+1}$};
    \node (6) at (10, 0) {$a_{\ell+2}$};
    \node at (11.75, 0) {$\cdots$};
    \node (w1) at (2, 0.22) {};
    \node (w2) at (2, -0.22) {};

    \node[dashed, draw, fit = (1) (2) (3) (4), ultra thick] (A) {};
    \node[draw, fit = (w1) (w2) (2) (3) (4) (5), ultra thick] (B) {};

    \node[below of=A, yshift=-1.5ex] (out1) {$b_{11}\cdots b_{k1}$};
    \draw[dashed, ->, ultra thick] (A.south) to (out1.north);
    \node[below of=B, yshift=-1.5ex] (out2) {$b_{12}\cdots b_{k2}$};
    \draw[->, ultra thick] (B.south) to (out2.north);

    \draw[ ->, ultra thick] (-0.5, -0.75) to (1.5, -0.75);

    \draw [thick, decorate, decoration={brace, amplitude=5pt, raise=4pt}]
    (1.6cm, 0.5) to node[above, yshift=0.3cm] {\small{sliding ``window''}}(8.4cm, 0.5);
  \end{tikzpicture}
  \caption{Diagram of the input-output process of convolutional codes:
    input bits $a_{1}a_{2}\dots a_{\ell}$ produce output bits
    $b_{11}b_{21}\dots b_{k1}$.}
  \label{fig:sliding-window}
\end{figure}

Comparing these two basic code design concepts under the assumption that
both codes produce a similar number of redundancy bits per block,
convolutional codes will be much longer and contain more redundancy. This is because
the input blocks overlap, and thus each bit is used multiple times when creating
the output.

This is why for convolutional codes, the idea
of simply deleting some of the output bits (in a systematic way) is very
practical (see~\cite[Chapter~4.1]{Johannesson-Zigangirov:2015:convolutional-coding})
as many error correction and performance properties are
preserved, and yet the code is made more efficient.
Let $k$, $n$, $r\in\N$. We refer to convolutional codes that generate
$k$ output bits per input window as $(k, 1)$ convolutional codes. By collecting
$n$ successive input blocks, encoding them to obtain $k n$ output bits
and deciding for all $1\leq i\leq k$ and $1\leq j\leq n$ whether the $i$-th bit
resulting from the $j$-th input window should be kept or deleted, a
\emph{punctured $(r, n)$ convolutional code} is created. Here, $r$ refers to the number of output bits that
are selected to remain, and $n$ indicates the number of input windows.

The systematic way of choosing the bits that are kept and form a punctured
$(r, n)$ convolutional code can be
represented by a binary $k\times n$ matrix
\[
  P = (p_{ij})_{\substack{1\leq i\leq k \\ 1\leq j\leq n}} \quad
  \text{ where } \quad
  p_{ij} = \begin{cases}
  1 & \text{ if $i$-th output bit of $j$-th block is kept,}\\
  0 & \text{ otherwise, }
  \end{cases}
\]
that has exactly $r$ entries that are 1.
This matrix is called the \emph{perforation pattern} of the code.

As an example, consider the perforation pattern
\[
  \begin{pmatrix}
    1 & 0 & 0 & 1 \\
    0 & 1 & 0 & 0 \\
    1 & 0 & 1 & 1 \\
  \end{pmatrix}.
\]
When applied to a $(3, 1)$ convolutional code, a punctured $(6, 4)$ code is
created which assembles its output by deleting the second bit in the first
output block, the first and third bit from the second block, the first two bits
from the third block, and the second bit from the fourth block.
Considering output similar to that in Figure~\ref{fig:sliding-window}, the output
$b_{11}b_{21}b_{31}b_{12}b_{22}b_{32}b_{13}b_{23}b_{33}b_{14}b_{24}b_{34}$ would be `punctured'
to form the new output $b_{11}b_{31}b_{22}b_{33}b_{14}b_{34}$.

In~\cite{Begin-Haccoun:1989:punctured-codes} two perforation
patterns are considered to be equivalent if they can be obtained from
each other by a cyclic permutation of columns, inducing an equivalence
relation on the set of perforation patterns (and therefore also on the
corresponding binary matrices). This is because two different
punctured codes that are constructed from a given convolutional code by
applying equivalent perforation patterns have exactly the same distance
properties, and thus the same error correction capabilities.

The OEIS~\cite{OEIS:2022} contains several sequences concerned with
the enumeration of non-equivalent perforation patterns of punctured convolutional codes,
namely \OEIS{A007223}--\OEIS{A007229}. According to these OEIS
entries, the corresponding sequences were investigated by Bégin and
presented at a conference in 1992~\cite{Begin:1992:enumeration-punctured-codes},
but not published otherwise.

Note once again, the objects of interest --- non-equivalent punctured
convolutional codes --- can be described by equivalence classes of
binary matrices. With the coding theoretic background in mind, we take a closer
look at these matrices to find the connection between them and our down-step
parameter.

\subsection{Equivalence classes of binary matrices}\label{sec:matrices}

In this section we construct a bijection between elevated $k_0$-Dyck paths (Definition~\ref{def:elevated})
with $n$ up-steps and $\{0,1\}$-matrices of dimension $(k+1)\times n$  with precisely $n+1$ many 1's,
considered up to cyclic permutation of columns.
This will imply that $\frac{1}{n} \binom{(k+1)n}{n+1}$ enumerates\footnote{Observe that 
this is compatible with Corollary~\ref{cor:1-marked-elevated} as
$\frac{1}{n} \binom{(k+1)n}{n+1} = \frac{k}{n+1} \binom{(k+1)n}{n}$.} these paths.
Therefore we re-establish the formula~\eqref{eq:GPW} due to~\cite{Gu-Prodinger-Wagner:2010:k-plane-trees},
while contributing a simple combinatorial interpretation of the factor~$\frac{1}{n}$ in the first formula.
For the codes, it follows that for each fixed~$k$, the sequence
$(s^{(k)}_{n, 0, 1})_{n \geq 1}$ (Corollary~\ref{cor:1-marked-elevated}) enumerates perforation patterns for
punctured convolutional $(n+1,n)$ codes: see \OEIS{A007226} for $k=2$,
\OEIS{A007228} for $k=3$,
\OEIS{A124724} for $k=4$;
these sequences are also columns of the rectangular array
\OEIS{A241262}.
As a by-product,
we obtain \textbf{a new result in the spirit of the Cycle Lemma},
and \textbf{a new structure enumerated by Catalan numbers} (see details below).

We introduce the following notions.
\begin{itemize}
\item Let $\EL^k_{n}$ be the family of elevated $k_0$-Dyck paths with $n$ up-steps
and $kn$ down-steps.
\item Let $\MA^k_{n}$ be the family of $\{0,1\}$-matrices of order
  $(k+1)\times n$ with $n+1$ many $1$'s and $kn-1$ many $0$'s.
\item We say that two matrices in $\MA^k_{n}$ are
  \textit{CPC-equivalent} if they can be obtained from each other by a
  \textit{cyclic permutation of columns} (it is easy to see that this is indeed an
  equivalence relation).  For example, the matrices
  $\big(\begin{smallmatrix} 1 & 0 & 1 \\ 1 & 1 & 0 \end{smallmatrix}\big)$,
  $\big(\begin{smallmatrix} 0 & 1 & 1 \\ 1 & 0 & 1 \end{smallmatrix}\big)$,
  $\big(\begin{smallmatrix} 1 & 1 & 0 \\ 0 & 1 & 1 \end{smallmatrix}\big)$
  are CPC-equivalent in $ \MA^1_{3}$.
\item Let $\overline \MA^k_{n}$ be the quotient set of $\MA^k_{n}$
  with respect to the CPC equivalence relation.
\end{itemize}
\begin{theorem}\label{thm:matrices_paths}
  For each $k\geq 1$, $n \geq 1$ we have
  \begin{equation}\label{eq:MAPA}
    \bigl| \EL^k_{n}\bigr|=  \frac{1}{n} \binom{(k+1)n}{n+1}.
  \end{equation}
\end{theorem}

\begin{proof}
  The proof consists of two parts: First, we show that $\bigl|\overline \MA^k_{n}\bigr|=\frac{1}{n} \binom{(k+1)n}{n+1}$;
  and then we construct a bijection between $\overline \MA^k_{n}$ and $\EL^k_{n}$.

  To begin with, we clearly have $\bigl|\MA^k_{n}\bigr|= \binom{(k+1)n}{n+1}$. Since in a matrix
  $A\in\MA^k_{n}$ we have $n+1$ many $1$'s distributed over $n$
  columns, it is not possible that a non-trivial cyclic permutation of
  the columns of $A$ will yield $A$ itself.
  (Indeed, if $d$ is the smallest positive number such that
  cyclic shifting of $A$ by $d$ columns yields $A$, then $A$ splits into $c=n/d$
  identical blocks, and we have $c\mid n$; but considering the number of $1$'s
  in each block we also have $c\mid n+1$, which necessarily yields $c=1$ and $d=n$.)
  Therefore each
  equivalence class consists of precisely $n$ matrices, and we have
  $\bigl|\overline \MA^k_{n}\bigr|=\frac{1}{n} \binom{(k+1)n}{n+1}$.

  Next, we consider two mappings, $f$ and $g$, from $\MA^k_{n}$ to the
  set of lattice paths with up-steps $(1,k)$ and down-steps $(1,-1)$:
  a specialization of the latter one will yield the desired bijection 
  between $\overline \MA^k_{n}$ and $\EL^k_{n}$.
  For $A \in \MA^k_{n}$, we define $f(A)$ and $g(A)$ as follows:
  \begin{itemize}
  \item Read the entries of $A$ column by column (read each column from top to bottom, and the columns from
    left to right). Starting
    at $(0, \, k)$, draw an up-step for each~$1$ entry, a down-step
    for each~$0$ entry.  This yields a path from $(0, \, k)$ to
    $((k+1)n, \, 2k+1)$: we denote this path by~$f(A)$.
  \item In $f(A)$, replace the first up-step by a down-step. The part
    of the path to the left of this step does not change, and the part
    to the right of this step is accordingly translated $k+1$ units
    downwards. At this point we obtain a path from $(0, \, k)$ to
    $((k+1)n, \, k)$.
  \item Shift this path to the left so that the right endpoint of the
    modified down-step will lie on the $y$-axis.  Finally, remove the
    part of the path to the left of the $y$-axis, and attach it to the
    right of the remaining path. This yields a path from $(0, \, j)$
    to $((k+1)n, \, j)$ for some~$j \leq k-1$: we denote this path
    by~$g(A)$.
  \end{itemize}

  See Figure~\ref{fig:ma2pa1} for two examples.

  For
  $\alpha=1, 2, \ldots, n+1$, we call a node in $f(A)$ (or in $g(A)$)
  with $x$-coordinate $(\alpha-1)(k+1)$ \textit{the $\alpha$-th
    principal node}: these nodes are marked in the figure.  Notice that
  the mapping $f$ preserves the cyclic structure of the matrix, in the
  following sense.  If $A'$ is obtained from $A$ by the cyclic
  permutation of columns such that the $\alpha$-th column of $A$ is the
  first column of $A'$, then one obtains $f(A')$ from $f(A)$ by cutting
  $f(A)$ at the $\alpha$-th principal node, exchanging the left and the
  right parts and joining them, and positioning the path so that its
  starting point is again $(0, k)$. This is also illustrated in
  Figure~\ref{fig:ma2pa1}.

  \begin{figure}
    \begin{center}
      \includegraphics[scale=0.85]{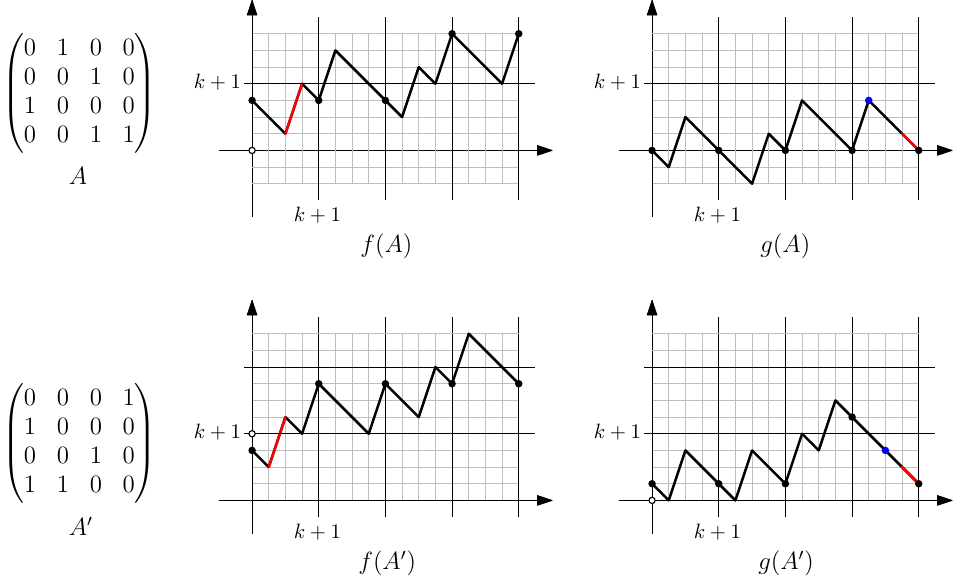}
      \caption{An illustration to the proof of Theorem~\ref{thm:matrices_paths}.
      The mappings $f$ and $g$ from matrices to paths are defined in the proof.
      The matrices $A$ and $A'$ are CPC-equivalent. $A'$ is valid,
      hence $g(A')$ is an elevated $k_0$-Dyck path.}
      \label{fig:ma2pa1}
    \end{center}
  \end{figure}

  It follows directly from the definitions of $f$ and $g$ that these mappings are
  injective.  By definition, $g(A)$ is a path from $(0, \, j)$ to
  $((k+1)n, \, j)$ for some~$j \leq k-1$.  However it is not always the
  case that $g(A)$ lies (weakly) above the $x$-axis (in fact, $j$
  can even be negative).  We say that $B \in \MA^k_{n}$ is \textit{valid} if
  $g(B)$ is an elevated $k_0$-Dyck path
  ($g(B) \in \EL^k_{n}$).  We will prove that each CPC-equivalence
  class in $\MA^k_{n}$ has a \textit{unique} valid
  representative. In other words: for each $A \in \MA^k_{n}$
  there is a unique valid $B$ obtained from $A$ by a (possibly trivial)
  cyclic permutation of columns.

  First we notice that $B \in \MA^k_{n}$ is valid if and only if $f(B)$
  is weakly above the line $y=k+1$ after its first up-step. Moreover,
  each principal node of $f(B)$ has $y$-coordinate of the form
  $\beta(k+1)-1$, and this implies one more equivalent condition:
  $B \in \MA^k_{n}$ is valid if and only if all the principal nodes of
  $f(B)$ (except for the starting point of the path) have
  $y$-coordinates of the form $\beta(k+1)-1$ with $\beta \geq 2$.

  Let $A \in \MA^k_{n}$. Let $\beta_0$ be the \textit{smallest} number
  such that some principal node (including the starting point) of $f(A)$ has $y$-coordinate
  $\beta_0(k+1)-1$.  Next, let
  \begin{equation*}
    \alpha_0 = \max\big\{\alpha : \big((\alpha-1)(k+1),\, \beta_0(k+1)-1\big) \text{ is a principal node}\big\},
  \end{equation*}
  and then consider $A'$ to be the matrix that is
  CPC-equivalent to~$A$ where the $\alpha_0$-th column of $A$ is the
  first column of~$A'$.  As a consequence, all the principal nodes in $f(A')$ except
  for the starting point have $y$-coordinates of the form $\beta(k+1)-1$
  with $\beta \geq 2$.  Therefore $A'$ is a valid matrix.

  It is also easy to show that in each CPC-equivalence class there is at
  most one valid matrix.  Indeed, assume that $A$ is valid. Then all the
  principal nodes of $f(A)$ (except for the starting point) have
  $y$-coordinates of the form $\beta(k+1)-1$ with $\beta \geq 2$.  If
  some non-trivial shift of $A$ is also valid, then, starting at some
  $\alpha < n+1$, all the principal nodes have $y$-coordinates of the
  form $\beta(k+1)-1$ with $\beta \geq 3$.  However, this contradicts
  the fact that $f(A)$ ends at $((k+1)n, \, 2k+1)$.

  Finally, we can reverse the process by taking $P\in\EL^k_{n}$,
  cutting it at the point with $y$-coordinate $k$ which lies at the
  last string of down-steps, modify/reconnect pieces accordingly,
  and read out the matrix.
  Such a cutting point (it is shown by blue colour in Figure~\ref{fig:ma2pa1})
  always exists, and it is not the end-point of the path:
  this follows from the fact
  that $P$ terminates at height $j$, where $0 \leq j \leq k-1$,
  and that the last string of down-steps in $P$ starts
  at height at least $k$. Thus we have a bijection
  between the set of valid matrices in $\MA^k_{n}$ and
  $\EL^k_{n}$. This shows that each CPC-equivalence class contains
  exactly one valid matrix, and completes the proof.
\end{proof}

\begin{remark}
  \begin{enumerate}[(a)]
  \item The claim that each CPC-equivalence class contains precisely
one valid matrix is a \textit{Cycle Lemma}-type result.
The classical Cycle Lemma~\cite{Dvoretzky-Motzkin:1947}
determines bounds on the number $N$ of those cyclic permutations of a given $\{0,1\}$-sequence
that satisfy the property: In each prefix, we have $\# 1 - \alpha \cdot \# 0>0$
(or $\# 1 - \alpha \cdot \# 0 \geq 0$) for a fixed number $\alpha$,
where $\# 0$ and $\# 1$ denote the number of zeros and ones in the
prefix, respectively.
It is possible to interpret such results in terms of lattice paths, 
where the inequality translates to some version of non-negativity of a path.
In our case the condition is that the sequence (extracted from $A$ column-by-column)
is such that in each prefix \textit{that contains at least one $1$}, we have
$k + (k\cdot\#1 - \#0) \geq k+1$ (see Figure~\ref{fig:ma2pa1}),
or, equivalently, $k\cdot\#1 - \#0 > 0$.
So, if we have a sequence of $n+1$ many $1$'s and $kn-1$ many $0$'s,
and partition it into $n$ consecutive blocks of $k+1$ entries, then
there is precisely one cyclic permutation \textit{of the blocks}
that satisfies the property above. 
To summarize, the difference between our result and the classical Cycle Lemma
is that in our case (1) the property holds \textit{with delay}
(meaning that the prefix has to contain at least one 1), and (2) we cyclically permute \textit{blocks}.
\item Recall that for the number of elevated $k_0$-Dyck paths we also have the formula
$\sum_{j=0}^{k-1}C_{n,j}$,
where $C_{n,j}$ enumerates elevated $k_0$-Dyck paths
of elevation $j$ ($0 \leq j \leq k-1$), or equivalently,
$k_j$-Dyck paths (see Corollary~\ref{cor:1-marked-elevated}).
In our bijection between elevated $k_0$-Dyck paths
and valid binary matrices, the elevation of a path corresponds to
$k$ minus the position of the first $1$ in the matrix
(converting it into a sequence as above). This yields a formula
for valid matrices with fixed position of the first $1$.
\item For $k = 1$, the CPC-equivalence classes are in bijection
with paths from $\EL_{n}^{1}$, which are just classical Dyck paths of length $2n$, 
and, thus, enumerated by Catalan numbers (see \OEIS{A000108} in the OEIS).
This means that we have found yet another combinatorial family enumerated 
by Catalan numbers. To the best of our knowledge, this interpretation has not
been mentioned before in literature.
\end{enumerate}
\end{remark}
\begin{corollary}[\textbf{New interpretation of Catalan numbers}]\label{cor:new-catalan}
  The family of equivalence classes of binary $2\times n$-matrices
  with $n+1$ many 1's, where two matrices are equivalent if they can
  be obtained from each other by a cyclic permutation of columns, is
  enumerated by the $n$-th Catalan number $C_{n}$.
\end{corollary}

The bijection between Dyck paths with $n$ up-steps and valid matrices
in $\MA_{n}^{1}$ can even be described in a simplified form. Given a
Dyck path of length $2n$, the associated valid $2\times n$ matrix is
constructed as follows: first, put a 1 in the first row and first
column. The remaining entries are filled column-by-column, from top to
bottom by reading the Dyck path from left to right and putting a 1 for
an up-step, and a 0 for a down-step. The last step of the path is
ignored. See Figure~\ref{fig:dyck-cpc} for an illustration.

\begin{figure}[ht]
  \centering\footnotesize
  \begin{minipage}[t]{0.18\linewidth}
    \centering
    \begin{tikzpicture}[scale=0.4]
      \drawlatticepath{1,1,1,-1,-1,-1}
    \end{tikzpicture}
    \[ \begin{pmatrix} 1 & 1 & 0 \\ 1 & 1 & 0 \end{pmatrix} \]
  \end{minipage}\hfill
  \begin{minipage}[t]{0.18\linewidth}
    \centering
    \begin{tikzpicture}[scale=0.4]
      \drawlatticepath{1,1,-1,1,-1,-1}
    \end{tikzpicture}
    \[ \begin{pmatrix} 1 & 1 & 1 \\ 1 & 0 & 0 \end{pmatrix} \]
  \end{minipage}\hfill
  \begin{minipage}[t]{0.18\linewidth}
    \centering
    \begin{tikzpicture}[scale=0.4]
      \drawlatticepath{1,1,-1,-1,1,-1}
    \end{tikzpicture}
    \[ \begin{pmatrix} 1 & 1 & 0 \\ 1 & 0 & 1 \end{pmatrix} \]
  \end{minipage}\hfill
  \begin{minipage}[t]{0.18\linewidth}
    \centering
    \begin{tikzpicture}[scale=0.4]
      \drawlatticepath{1,-1,1,1,-1,-1}
    \end{tikzpicture}
    \[ \begin{pmatrix} 1 & 0 & 1 \\ 1 & 1 & 0 \end{pmatrix} \]
  \end{minipage}\hfill
  \begin{minipage}[t]{0.18\linewidth}
    \centering
    \begin{tikzpicture}[scale=0.4]
      \drawlatticepath{1,-1,1,-1,1,-1}
    \end{tikzpicture}
    \[ \begin{pmatrix} 1 & 0 & 0 \\ 1 & 1 & 1 \end{pmatrix} \]
  \end{minipage}
  \caption{The $C_{3} = 5$ Dyck-paths with $3$ up-steps, each with their associated valid matrix from $\MA_{3}^{1}$. }
  \label{fig:dyck-cpc}
\end{figure}
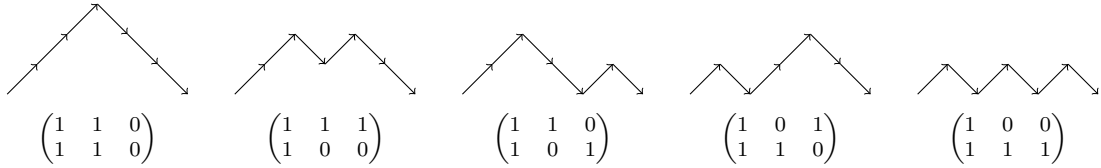

The above investigations were for the case where we have precisely
$n+1$ many $1$'s in the binary matrices. 
However, we can also handle the case where we have an
arbitrary number of 1's in the matrix by making use of the Pólya
theory framework (see~\cite{Polya:1937:kombin-anzah} for Pólya's
original work and~\cite[Remark I.60]{Flajolet-Sedgewick:ta:analy} for
a useful statement in terms of generating functions and combinatorial
classes). This has also been observed
in~\cite{Begin:1992:enumeration-punctured-codes}.

\begin{proposition}
  Let $k$, $n\in \mathbb{N}$ be fixed positive integers and let $z$ be a
  symbolic variable. Then with $\varphi$ as Euler's totient function, the generating polynomial
  \begin{equation}\label{eq:matrix-eq-genpoly}
    M_{k, n}(z) = \frac{1}{n} \sum_{d \mid n} \varphi(d) (1 + z^d)^{(k+1) n / d}
  \end{equation}
  enumerates all equivalence classes of $(k+1)\times n$ binary matrices under the
  CPC-equivalence where the power of $z$ corresponds to
  the number of 1's.
\end{proposition}
\begin{proof}
  Let $x_1$, \ldots, $x_n$ be symbolic variables. It is a well-known
  fact~\cite[Remark~I.60]{Flajolet-Sedgewick:ta:analy} that the cycle index 
  of the group of cyclic permutations of length $n$ is given by
  \begin{equation}\label{eq:cycle-indicator-Cn}
    Z(x_1, x_2, \dots, x_n) = \frac{1}{n} \sum_{d \mid n} \varphi(d) x_d^{n/d}.
  \end{equation}
  The Pólya--Redfield theorem assumes the following setting: We consider a family of
  combinatorial objects (binary matrices of dimension $(k+1)\times n$) that are
  constructed from $n$ other ``building blocks'' (the columns of the binary
  matrices). We wish to enumerate the composed objects, given that two of them
  are considered to be the same if their building blocks are a cyclic shift of each
  other. Then, the theorem asserts that the generating function enumerating the composed
  objects is given by
  \[ Z(B(z), B(z^2), \ldots, B(z^n)), \]
  where $B(z)$ is the generating function enumerating the ``building block'' objects.

  Translating this to our setting we find that the $(k+1)\times 1$ binary columns are
  enumerated by $B(z) = (1 + z)^{k+1}$ (as for every entry we can either choose 1, which
  contributes $z$, or we can choose 0, which contributes $z^0 = 1$). Thus, applying the
  Pólya--Redfield theorem to this situation yields~\eqref{eq:matrix-eq-genpoly}.
\end{proof}

From this enumerating polynomial, formulas for the number of equivalence classes
with an arbitrary number of 1's can be derived by extracting the coefficient
of the corresponding monomial.

\section{Down-step statistics: A generating function approach}\label{sec:generating-functions}

In this section we will consider the symbolic decomposition of $k_t$-Dyck paths, making
use of the symbolic method (see \cite[Chapter I.1]{Flajolet-Sedgewick:ta:analy}).
From this, we then calculate down-step statistics using bivariate
generating functions with the variables $x$ and $u$ to count the number of up-steps
and the number of down-steps, respectively.

Recall that the definition of $\mathcal{S}_t$, the class of $k_t$-Dyck paths,
is given in Definition~\ref{def:k_t-Dyck}. 
Let $k\geq 2$ and $0 \leq t \leq k$. The symbolic decomposition of $k_t$-Dyck paths is given by
\begin{equation}\label{eq:symbolic-decomposition}
\mathcal{S}_t = \varepsilon + \sum_{i = 0}^t \big(\{\down\}^{i}\times \{\up\}
\times \mathcal{S}_{k-1-i}\times\{\down\}^{k-i}\times \mathcal{S}_{t}\big),
\end{equation}
where $\down$ represents a $(1, -1)$ step, $\up$ represents a $(1,
k)$ step, $\varepsilon$ represents the empty path, and by convention we 
set $\mathcal{S}_{-1} = \varepsilon$. This
is obtained via a first return decomposition (with the return being to
$y = 0$) of the path, and illustrated
in Figure~\ref{fig:kDyckdec}. Here we have $i$ down-steps, where $0 \leq i \leq t$,
followed by an up-step, bringing the height to $k-i$. In order to not introduce any returns
in this part of the path, we have an arbitrary $k_{k-i-1}$-Dyck path, again ending at a height of $k-i$. Finally, we introduce
the first return to $y=0$ of the path with $k-i$ down-steps, and the remainder of the path is an
arbitrary $k_t$-Dyck path.

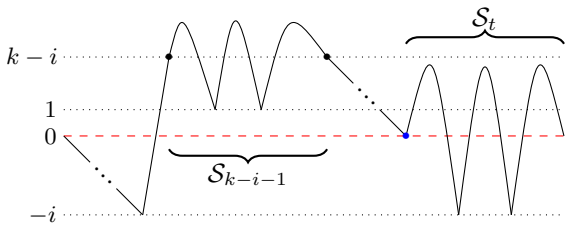
\begin{figure}[ht]
  \begin{tikzpicture}[scale = 0.35]
    \draw [-, dotted] (0, 1) to (19, 1);
    \draw [-, dotted] (0, 3) to (19, 3);
    \draw [-, dotted] (0, -3) to (19, -3);
    \draw [-, red, dashed] (0, 0) to (19, 0);

    \draw [-] (0, 0) to (1, -1);
    \node at (1.25, -1.25) {.};
    \node at (1.5, -1.5) {.};
    \node at (1.75, -1.75) {.};
    \draw [-] (2, -2) to (3, -3);

    \draw [-] (3, -3) to (4, 3);

    \draw (4, 3) .. controls (4.5, 5) .. (5.75, 1);
    \draw (5.75, 1) .. controls (6.5, 5.5) .. (7.5, 1);
    \draw (7.5, 1) .. controls (8.5, 5) .. (10, 3);

    \draw [-] (10, 3) to (11, 2);
    \node at (11.25, 1.75){.};
    \node at (11.5, 1.5){.};
    \node at (11.75, 1.25){.};
    \draw [-] (12, 1) to (13, 0);

    \draw (13, 0) .. controls (14, 4) .. (15, -3);
    \draw (15, -3) .. controls (16, 4.5) .. (17, -3);
    \draw (17, -3) .. controls (18, 4) .. (19, 0);

    \node at (4, 3) {\tiny{$\bullet$}};
    \node at (10, 3) {\tiny{$\bullet$}};
    \node[color=blue] at (13, 0) {\tiny{$\bullet$}};

    \node at (-0.5, 0) {\small{$0$}};
    \node at (-0.5, 1) {\small{$1$}};
    \node at (-1.2, 3) {\small{$k-i$}};
    \node at (-0.8, -3) {\small{$-i$}};

    \draw [thick, decorate, decoration={brace, amplitude=5pt, mirror, raise=4pt}]
    (4cm, -0.1) to node[below,yshift=-0.2cm] {$\mathcal{S}_{k-i-1}$}(10cm, -0.1);
    \draw [thick, decorate, decoration={brace, amplitude=5pt, raise=4pt}]
    (13cm, 3.1) to node[above,yshift= 0.2cm] {$\mathcal{S}_{t}$}(19cm, 3.1);
  \end{tikzpicture}
  \caption{The symbolic decomposition of a $k_t$-Dyck path for fixed $i$.}
  \label{fig:kDyckdec}
\end{figure}

\begin{remark}
  Note that this decomposition does not work for $t > k$, as for $k < i \leq t$ one up-step $(1, k)$
  would not be enough to return to or above the $x$-axis, and thus a decomposition based on 
  the first return to the $x$-axis would be much more involved, as has been outlined in \cite{Prodinger:2019:negative-boundary}.
\end{remark}

This translates into the following system of functional equations for $0 \leq t \leq k$, where
$S_t(x)$ is the generating function for $k_t$-Dyck paths with $x$ counting
the number of up-steps,
\begin{equation*}
  S_t(x) = 1 + xS_t(x)\sum_{i=0}^tS_{k-i-1}(x)\qquad \text{ and
  }\qquad S_{-1}(x) = 1.
\end{equation*}
With the substitution as used in \cite[Section 3.3]{Gessel:2016:lagrange}
to deal with the Fuss--Catalan family,
\begin{equation}\label{eq:substitution}
  x = z(1-z)^{k}
\end{equation}
we can show (see~\cite[Chapter~6]{Selkirk:2019:MSc}) that the generating function
for $k_t$-Dyck paths simplifies to the expression $S_t(x) = (1-z)^{-t-1}$. 
Throughout this section this substitution will be used to simplify expressions and thus extract
coefficients in an easier manner.

\subsection{The number of down-steps at the end of $k_t$-Dyck paths}

Before calculating down-step statistics generally, we first
consider a special case which will be used in all other calculations of the total number of down-steps: the total
number of down-steps at the end of all $k_t$-Dyck paths of a given length.
\begin{proposition}\label{thm:down-steps-end}
  The total number of down-steps at the end of all $k_t$-Dyck paths of length
  $(k+1)n$ is equal to
  \begin{equation*}
    s_{n,t,n} =
    \frac{t+1}{n+1}\binom{(k+1)(n+1)+t}{n} -
    \frac{(t+1)^2}{n}\binom{(k+1)n+t}{n-1},
  \end{equation*}
  where $s_{n,t,r}$ is used as in Definition~\ref{def:downsteps}.
\end{proposition}
Note that these numbers are, as expected, the same as those of
\eqref{eq:last-downsteps} in Proposition~\ref{prop:last-downsteps},
however, using this generating function approach allows
us to later obtain the variance --- a statistic which is more difficult
to obtain via the bijective approach. Interestingly, the only pre-existing
combinatorial interpretation of these numbers, related to
non-crossing trees, occurred for $k=2$ and $t=1$ as entry
\OEIS{A030983} in the OEIS. We have added entries for additional parameter
combinations as \OEIS{A334609}--\OEIS{A334612}, \OEIS{A334680},
\OEIS{A334682}, \OEIS{A334719}.

Before proving this, we will prove a lemma that will be used to extract coefficients from generating functions under the substitution
$x = z(1-z)^k$. This will be done to prevent repetition of very similar calculations.
\begin{lemma}\label{lem:cauchy}
  For integers $n$, $k$, $t$, $a$, $b$, $c$, and $d$, we have
  \begin{equation*}
    [x^n]\frac{1}{z^a(1-z)^{bk+ct+d}} = \frac{k(b-a)+ct+d}{n+a}\binom{(k+1)n+kb+ct+d+a-1}{n+a-1}.
  \end{equation*}
\end{lemma}
\begin{proof}
  We will extract coefficients by means of Cauchy's
integral formula (see \cite[Theorem IV.4]{Flajolet-Sedgewick:ta:analy}).
Let $\gamma$ be a small contour that winds around the origin once,
and let $\tilde\gamma$ be the image of $\gamma$ under the substitution
$x := z(1-z)^{k}$ from~\eqref{eq:substitution}. Then
\begin{align*}
  [x^{n}]\frac{1}{z^a(1-z)^{bk+ct+d}}
  &= \frac{1}{2\pi i} \oint_{\gamma} \frac{1}{z^a(1 - z)^{bk+ct+d}}
    \frac{1}{x^{n+1}}\, dx\\
  & = \frac{1}{2\pi i}\oint_{\tilde\gamma} \frac{\frac{1}{z^a(1-z)^{bk+ct+d}}}{(z(1-z)^k)^{n+1}}(1-(k+1)z)(1-z)^{k-1}\, dz\\
                                      & = \frac{1}{2\pi i}\oint_{\tilde\gamma} \frac{1}{z^{n+a+1}(1-z)^{k(n+b)+ct+d+1}}\, dz\\
                                      & \qquad - \frac{k+1}{2\pi i}\oint_{\tilde\gamma} \frac{1}{z^{n+a}(1-z)^{k(n+b)+ct+d+1}}\, dz\\
                                      & = [z^{n+a}]\frac{1}{(1-z)^{k(n+b)+ct+d+1}} - (k+1)[z^{n+a-1}]\frac{1}{(1-z)^{k(n+b)+ct+d+1}}\\
                                      & = \frac{k(b-a)+ct+d}{n+a}\binom{(k+1)n+kb+ct+d+a-1}{n+a-1}.\qedhere
\end{align*}
\end{proof}

\begin{proof}[of Proposition~\ref{thm:down-steps-end}]
We use $u$ to count the number of down-steps after the last up-step, and $x$ to count
the number of up-steps in the paths. With Figure~\ref{fig:kDyckdec} as reference: if the final
$k_{t}$-Dyck path (denoted by $\mathcal{S}_{t}$ in Figure~\ref{fig:kDyckdec}) is empty, then the contribution to the
sequence of down-steps at the end of the path
comes from the $k-i$ down-steps as well as the final down-steps
from the $k_{k-i-1}$-Dyck path. Otherwise, just the final down-steps
of the $k_{t}$-Dyck path contribute towards the total number of down-steps.
For $0 \leq t \leq k$ this leads to the system $S_{-1}(x, u) = 1$ and
\begin{equation}\label{eq:en-down-gf}
  S_t(x, u) = 1 + x\sum_{i=0}^tS_{k-i-1}(x, u)u^{k-i} +
  x\sum_{i=0}^tS_{k-i-1}(x, 1)(S_t(x, u) - 1),
\end{equation}
where $S_{k-i-1}(x, 1) = (1-z)^{-k+i}$.
We use the notation $\partial_u$ to represent the partial derivative with respect to $u$.
Differentiating with respect to $u$, setting $u=1$, and substituting $x =
z(1-z)^k$ gives
\begin{align*}
  \partial_u S_t(x, u)\big|_{u=1}
  & = z(1-z)^k\sum_{i=0}^t \partial_u S_{k-i-1}(x, u)\big|_{u=1} - (k-t)(1-z)^{t+1} + k+1\\
  & \quad - \frac{1}{z} + \frac{(1-z)^{t+1}}{z} + \partial_u S_t(x, u)\big|_{u=1}(1-(1-z)^{t+1}).
\end{align*}
Collecting the $\partial_u S_t(x, u)\big|_{u=1}$ terms on the left-hand side, we obtain\footnote{For specific values these computations are carried out in the interactive
SageMath worksheet at \url{https://gitlab.aau.at/behackl/kt-dyck-downstep-code}.}
\begin{equation}\label{eq:en-down-diff}
  \partial_u S_t(x, u)\big|_{u=1}
  = \frac{z}{(1-z)^{t-k+1}}\sum_{i=0}^t \partial_u S_{k-i-1}(x,
    u)\big|_{u=1} - k+t + \frac{k+1}{(1-z)^{t+1}} - \frac{1 -
    (1-z)^{t+1}}{z(1-z)^{t+1}}.
\end{equation}
It can be proved by induction on $t$ and substitution into the above equation that the
generating function for the total number of down-steps at the end of a $k_t$-Dyck path is given by
\begin{equation*}
  \partial_u S_t(x, u)\big|_{u=1} = -\frac{t+1}{(1-z)^{t+1}} + \frac{1}{z(1-z)^{k+t+1}} - \frac{1}{z(1-z)^k}.
\end{equation*}
Uniqueness of this solution can be shown using bootstrapping. Since everything on the right-hand
side in \eqref{eq:en-down-diff} except for the sum of partial derivatives can be expanded in terms of $z$,
using the initial estimate $\partial_uS_t(x, u)\big|_{u=1} = O(z)$ and expanding leads to the new estimate
\begin{equation*}
  \partial_uS_t(x, u)\big|_{u=1} = (k+1)(t+1)z + \frac{(t+1)t}{2}z - (t+1)^2z + O(z^2),
\end{equation*}
and this can be continued indefinitely. Throughout this section there will be several expressions which can be verified to
be a solution to a given functional equation, and with similar bootstrapping the uniqueness can be shown. Finally, applying
Lemma~\ref{lem:cauchy} for each of the terms in the generating function, we obtain
\begin{equation*}
  [x^n]\frac{1}{(1-z)^{t+1}} = \frac{t+1}{n}\binom{(k+1)n+t}{n-1}; \qquad [x^n]\frac{1}{z(1-z)^{k+t+1}} = \frac{t+1}{n+1}\binom{(k+1)(n+1)+t}{n},
\end{equation*}
from which the result follows. Note that we do not need to compute coefficients of $1/z(1-z)^k$ nor
the positive integer powers of this term,
since these are equal to $1/x$ and powers thereof and thus do not contribute to the coefficients.
\end{proof}

\begin{proposition}\label{prop:r=n-2d}
  The coefficients of the second derivative of $S_t(x, u)$ with respect to $u$ at $u = 1$ are given by
  \begin{align*}
    [x^{n}]\partial_u^2S_t(x, u)\Big|_{u=1} & = \frac{2(t+1)}{n+2}\binom{(k+1)(n+2)+t}{n+1}  + \frac{(t+1)^2(t+2)}{n}\binom{(k+1)n+t}{n-1}\\
    & \quad - \frac{2(t+k+2)(t+1)}{n+1}\binom{(k+1)(n+1)+t}{n}.
  \end{align*}
\end{proposition}

The proof of the above proposition can be found in Appendix~\ref{ap:App-A}. 

\subsection{The number of down-steps between pairs of up-steps}

Here we again use the parameter $u$ to count the number of down-steps
in between pairs of up-steps. Let $S_{t, r}(x, u)$ be the bivariate
generating function for the number of down-steps between the $r$-th
and the $(r+1)$-th up-steps, or in the case where an $(r+1)$-th up-step
does not exist, after the $r$-th up-step. Using this we can easily set up a system
of functional equations using~\eqref{eq:symbolic-decomposition} to
determine $S_{t, r}(x, u)$. The case where $r = 0$ is given
by $S_{-1,0}(x, u) = 1$ and
\begin{equation}\label{eq:r=0}
  S_{t,0}(x, u) = 1 + xS_t(x, 1)\sum_{i=0}^tu^{i}S_{k-i-1}(x, 1).
\end{equation}
\begin{theorem}\label{thm:r=0}
  Let $n\geq 1$. The total number of down-steps before the first up-step
  in $k_t$-Dyck paths of length $(k+1)n$ is given by
  \begin{equation*}
    s_{n,t,0} = \frac{k}{n+1}\binom{(k+1)n}{n} - \frac{k-t}{n+1}\binom{(k+1)n+t}{n}.
  \end{equation*}
\end{theorem}
Again, by definition these numbers are equal to those
from~\eqref{eq:downstep-recurrence:ell0}, and have been added to the
OEIS as entries \OEIS{A001764},
\OEIS{A002293}, \OEIS{A002294}, \OEIS{A334785}--\OEIS{A334787}.
Together, the bijective and
generating function approaches yield the following interesting
summation identity related to~\eqref{eq:sum-of-k_t},
\begin{equation}\label{eq:FC_telescope}
 \sum_{j=0}^{t-1} C_{n, j} = \frac{k}{n+1}\binom{(k+1)n}{n} -
  \frac{k-t}{n+1}\binom{(k+1)n+t}{n}.
\end{equation}
This identity can be interpreted combinatorially by counting
$k_t$-Dyck paths. On the left-hand side we have the total number of
$k_j$-Dyck paths of length $(k+1)n$ for $0 \leq j \leq t-1$. By
Observation~\ref{obs:elevated}, these are in bijection with elevated
$k_0$-Dyck paths of elevation $j$ where $0 \leq j \leq k-1$, which we
know from \eqref{eq:GPW} to be enumerated by
\begin{equation*}
  \frac{k}{n+1}\binom{(k+1)n}{n},
\end{equation*}
minus the elevated $k_0$-Dyck paths of elevation $j$ where $t \leq j \leq k-1$.
It can be
verified by symbolic summation that summing the number of elevated $k_0$-Dyck paths over $t \leq j \leq k-1$ gives
\begin{equation}\label{eq:FC-summation}
\sum_{j = t}^{k-1}\frac{j+1}{(k+1)n+j+1}\binom{(k+1)n+j+1}{n} = \frac{k-t}{n+1}\binom{(k+1)n+t}{n}.
\end{equation}
With this, we have that the summation on the left-hand side of \eqref{eq:FC_telescope} is equal to the right-hand side.

Alternatively, \eqref{eq:FC_telescope} can be proved by considering
the following identity. This is a consequence
of~\eqref{eq:FC-summation} or can be verified by
algebraic manipulation.
\begin{equation*}
C_{n, j} = \frac{k-j}{n+1}\binom{(k+1)n+j}{n} -
  \frac{k-(j+1)}{n+1}\binom{(k+1)n+(j+1)}{n}.
\end{equation*}
Summing over $0 \leq j \leq t-1$ yields the following telescoping sum
\begin{align*}
  \sum_{j = 0}^{t-1}C_{n, j} & = \sum_{j = 0}^{t-1}\frac{k-j}{n+1}\binom{(k+1)n+j}{n} - \sum_{j=0}^{t-1}\frac{k-(j+1)}{n+1}\binom{(k+1)n+(j+1)}{n}\\
  & = \frac{k}{n+1}\binom{(k+1)n}{n} - \frac{k-t}{n+1}\binom{(k+1)n+t}{n}.
\end{align*}

\begin{proof}[of Theorem~\ref{thm:r=0}]
The system in \eqref{eq:r=0} can be further simplified with the substitution $x = z(1-z)^k$ and the expression $S_t(x) = (1-z)^{-t-1}$:
\begin{equation}\label{eq:Kt0}
  S_{t,0}(x, u) = 1 + \frac{z(1-z)^k}{(1-z)^{k+t+1}}\sum_{i=0}^{t}u^i(1-z)^i =
  1 + \frac{z}{(1-z)^{t+1}}\frac{1-u^{t+1}(1-z)^{t+1}}{1-u+uz}.
\end{equation}
Differentiating this system and setting $u=1$ yields
\begin{align*}
  \partial_u S_{t,0}(x, u)\big|_{u=1}
  & = \frac{z}{(1-z)^{t+1}}\Big(\frac{-(t+1)(1-z)^{t+1}z + (1-z)(1-(1-z)^{t+1})}{z^2}\Big)\\
  & = -(t+1) + \frac{1}{z(1-z)^{t}} - \frac{1-z}{z} = \frac{1}{z(1-z)^t} - \frac{1}{z} - t.
\end{align*}
Applying Lemma \ref{lem:cauchy} to extract coefficients yields the
expression given in the theorem.
\end{proof}

\begin{proposition}\label{prop:r=0-2d}
  Let $n \geq 1$. The coefficients of the second derivative (with respect to $u$) of $S_{t, 0}(x, u)$ are
  \begin{align*}
    [x^{n}]\partial_u^2S_{t, 0}(x, u)\big|_{u=1} & = \frac{2(t-2k-1)}{n+2}\binom{(k+1)n+t}{n+1} + \frac{2(t-1)k}{n+1}\binom{(k+1)n}{n}\\
    & \quad + \frac{4k}{n+2}\binom{(k+1)n+1}{n+1}.
  \end{align*}
\end{proposition}
\begin{proof}
  Differentiating the expression in \eqref{eq:Kt0} twice with respect to $u$ and setting $u = 1$ gives
  \begin{align*}
    \partial_u^2S_{t, 0}(x, u)\big|_{u=1} & = \frac{2}{z^2(1-z)^{t-1}} - \frac{t(t-1)z^2 + 2(t-1)z + 2}{z^2}
  \end{align*}
  The coefficients of this function can then be extracted using Lemma \ref{lem:cauchy} to obtain the result.
\end{proof}

For $r=1$, with reference to Figure \ref{fig:kDyckdec} and the symbolic decomposition \eqref{eq:symbolic-decomposition}
we must consider two cases:
\begin{itemize}
\item the sequence of down-steps before the second up-step lies
  entirely within the path from $\mathcal{S}_{k-i-1}$, and
\item the path from $\mathcal{S}_{k-i-1}$ is the empty path and so the sequence of down-steps
  before the second up-step includes the $k-i$ down-steps after the path from
  $\mathcal{S}_{k-i-1}$ as well as the down-steps before the first
  up-step in the path from $\mathcal{S}_t$.
\end{itemize}
This gives the system of functional equations for $0 \leq t \leq k$, setting $S_{-1,1}(x, u) = 1$,
\begin{equation*}
  S_{t,1}(x, u) = 1 + x\sum_{i=0}^t(S_{k-i-1,0}(x, u)-1)S_t(x, 1) + x\sum_{i=0}^tu^{k-i}S_{t,0}(x, u).
\end{equation*}
\begin{theorem}\label{thm:r=1}
 The total number of down-steps between the first and second up-steps in
 $k_t$-Dyck paths of length $(k+1)n$ is equal to
 \begin{equation}\label{eq:r=1:nt1}
   s_{n,t,1} = \frac{k}{n+1}\binom{(k+1)n}{n} - \frac{k-t}{n+1}\binom{(k+1)n+t}{n} +
   \frac{k(t+1)}{n}\binom{(k+1)(n-1)}{n-1}.
 \end{equation}
\end{theorem}
\begin{proof}
With the substitution $x = z(1-z)^k$ and using the generating
function for $S_{t,0}(x, u)$ in terms of $z$ and $u$ as given in~\eqref{eq:Kt0} we have that
\begin{align}\label{eq:S_{t, 1}}
  S_{t,1}(x, u)
  & = 1 + z(1-z)^k\frac{1}{(1-z)^{t+1}}\Big(\frac{1-(1-z)^{t+1}}{(1-z)^k(1-u+uz)} - \frac{zu^k}{1-u+uz}\sum_{i=0}^tu^{-i}\Big)\nonumber\\
  & \quad + z(1-z)^k\sum_{i=0}^tu^{k-i}\Big(1 + \frac{z}{(1-z)^{t+1}}\frac{1-u^{t+1}(1-z)^{t+1}}{1-u+uz}\Big)\nonumber\\
  & = 1 + \frac{z(1-(1-z)^{t+1})}{(1-z)^{t+1}(1-u+uz)} + \Big(z(1-z)^ku^k - \frac{z^2u^{k+t+1}(1-z)^k}{1-u+uz}\Big)\sum_{i=0}^tu^{-i}.
\end{align}
Differentiating this with respect to $u$ results in
\begin{align*}
  \partial_u S_{t,1}(x, u)
  & = \frac{z(1-(1-z)^{t+1})}{(1-z)^{t}(1-u+uz)^{2}} - \Big(z(1-z)^ku^k - \frac{z^2u^{k+t+1}(1-z)^k}{1-u+uz}\Big)\sum_{i=0}^tiu^{-i-1}\\
  & \quad + \Big(kz(1-z)^ku^{k-1} -
    \frac{z^2u^{k+t}(1-z)^k((k+t)(1-u+uz) + 1)}{(1-u+uz)^2}\Big)\sum_{i=0}^tu^{-i},
\end{align*}
and finally setting $u=1$ gives the expression
\begin{align*}
  \partial_u S_{t,1}(x, u)\big|_{u=1}
  & = \frac{1-(1-z)^{t+1}}{z(1-z)^{t}} - (t+1)(1-z)^k(tz + 1).
\end{align*}
Extracting the $n$-th coefficient (for $n \geq 2$) gives the
formula for the total number of down-steps between the first and second
up-steps in $k_t$-Dyck paths of length $(k+1)n$ as given in~\eqref{eq:r=1:nt1}.
\end{proof}

\begin{proposition}\label{prop:r=1-2d}
  Let $n\geq 2$. The coefficients of the second derivative with respect to $u$ of $S_{t, 1}(x, u)$ are
  \begin{align*}
    [x^{n}]\partial_u^2S_{t, 1}(x, u)\big|_{u=1} & = \frac{2(t-2k-1)}{n+2}\binom{(k+1)n+t}{n+1} + \frac{4(k+1)}{n+2}\binom{(k+1)n-1}{n+1}\\
    & \quad + \frac{k(t+1)(2k+t-2)}{n}\binom{(k+1)(n-1)}{n-1} + \frac{4(t+1)k}{n+1}\binom{(k+1)n-k}{n}.
  \end{align*}
\end{proposition}
\begin{proof}
  Differentiating the expression in \eqref{eq:S_{t, 1}} twice with respect to $u$ and setting $u = 1$ yields
  \begin{align*}
    \partial_u^2S_{t, 1}(x, u)\big|_{u=1} & = \frac{2}{z^2(1-z)^{t-1}} - \frac{2(1-z)^2}{z^2} - \frac{(1-z)^k(t+1)(t(2k-1)z^2 + (2k+t-2)z + 2)}{z}.
  \end{align*}
  Extracting coefficients of each term via Lemma \ref{lem:cauchy} gives the result in the proposition.
\end{proof}

When $r \geq 2$ the system of functional equations as well as the calculations
become trickier --- here there are now three possible configurations with
reference to the symbolic decomposition in \eqref{eq:symbolic-decomposition}:
the sequence of down-steps between the $r$-th and $(r+1)$-th up-steps
\begin{enumerate}
\item lies entirely within the path from $\mathcal{S}_{k-i-1}$,
\item lies entirely within the path from $\mathcal{S}_{t}$, and
\item\label{item:config3} consists of the final down-steps of the path from $\mathcal{S}_{k-i-1}$,
  the $k-i$ intermediate down-steps, and the initial down-steps of
  the path from $\mathcal{S}_t$ (illustrated in Figure \ref{fig:third-bullet}).
\end{enumerate}
\begin{figure}[ht]
  \begin{tikzpicture}[scale = 0.35]
    \draw [-, dotted] (0, 1) to (20, 1);
    \draw [-, dotted] (0, 3) to (20, 3);
    \draw [-, dotted] (0, -3) to (20, -3);
    \draw [-, red, dashed] (0, 0) to (20, 0);

    \draw [-] (0, 0) to (1, -1);
    \node at (1.25, -1.25) {.};
    \node at (1.5, -1.5) {.};
    \node at (1.75, -1.75) {.};
    \draw [-] (2, -2) to (3, -3);

    \draw [-] (3, -3) to (4, 3);

    \draw (4, 3) .. controls (4.5, 6) .. (5.75, 1);
    \draw (5.75, 1) .. controls (6.5, 6.5) .. (7.5, 1);

    \draw [-] (7.5, 1) to (8, 5);
    \draw [-, thick] (8, 5) to (10, 3);

    \draw [-, thick] (10, 3) to (11, 2);
    \node at (11.25, 1.75){.};
    \node at (11.5, 1.5){.};
    \node at (11.75, 1.25){.};
    \draw [-, thick] (12, 1) to (13, 0);

    \node at (13, 0) {\tiny{$\bullet$}};
    \draw [-, thick] (13, 0) to (14, -1);

    \draw (14, -1) .. controls (15, 4) .. (16, -3);
    \draw (16, -3) .. controls (17, 4.5) .. (18, -3);
    \draw (18, -3) .. controls (19, 4) .. (20, 0);

    \node at (4, 3) {\tiny{$\bullet$}};
    \node at (10, 3) {\tiny{$\bullet$}};

    \node at (-0.5, 0) {\small{$0$}};
    \node at (-0.5, 1) {\small{$1$}};
    \node at (-1.2, 3) {\small{$k-i$}};
    \node at (-0.8, -3) {\small{$-i$}};

    \draw [thick, decorate, decoration={brace, amplitude=2pt, mirror, raise=4pt}]
    (3cm, -3) to node[below, yshift = -0.35cm] {\small{$x$}}(3.9cm, -3);
    \draw [thick, decorate, decoration={brace, amplitude=5pt, mirror, raise=4pt}]
    (4.1cm, -3) to node[below, yshift = -0.35cm] {\small{$S_{k-i-1}(x, u)$}}(9.9cm, -3);
    \draw [thick, decorate, decoration={brace, amplitude=5pt, mirror, raise=4pt}]
    (10.1cm, -3) to node[below, yshift = -0.3cm] {\small{$u^{k-i}$}}(12.9cm, -3);
    \draw [thick, decorate, decoration={brace, amplitude=5pt, mirror, raise=4pt}]
    (13.1cm, -3) to node[below, yshift = -0.3cm] {\small{$S_{t,0}(x, u)$}}(20cm, -3);
  \end{tikzpicture}
  \caption{An illustration of Configuration~\ref{item:config3}.}
  \label{fig:third-bullet}
\end{figure}
This leads to the general functional equation for $r \geq 2$,
\begin{align}
  S_{t,r}(x, u)
  & = 1 + xS_t(x, 1)\sum_{i=0}^t\sum_{n \geq r}x^n\big([x^n]S_{k-i-1,r-1}(x, u)\big)\label{eq:term-one}\\
  & \quad + x\sum_{i=0}^t\sum_{j = 0}^{r-2}x^{j}\big([x^{j}]S_{k-i-1}(x, 1)\big)S_{t,r-1-j}(x, u)\label{eq:term-two}\\
  & \quad + x\sum_{i=0}^tx^{r-1}\big([x^{r-1}]S_{k-i-1}(x, u)\big)u^{k-i}S_{t,0}(x, u).\label{eq:term-three}
\end{align}
Differentiating this with respect to $u$ and then setting $u = 1$, results in the following simplifications:
\begin{itemize}
  \item The sum in the line marked \eqref{eq:term-three} splits into three sums after differentiating and the
  differentiated term containing $\partial_u S_{t, 0}(x, u)$ combines with the differentiated sum from the line marked
  \eqref{eq:term-two} to make the sum over $0 \leq j \leq r-1$.
  \item Since $S_{k-i-1, r-1}(x, u)$ counts the number of down-steps between the $(r-1)$-th and $r$-th up-steps, and
  thus for $n < r-1$ there are no $u$ terms in the coefficients $[x^n]S_{k-i-1,r-1}(x, u)$. One of the resulting three
  differentiated terms of the term in line marked \eqref{eq:term-three} combines with the sum in the line marked
  \eqref{eq:term-one}, and it follows that
        \begin{equation*}
          \partial_u \sum_{n \geq r-1}x^n\big([x^n]S_{k-i-1,r-1}(x, u)\big) = \partial_u S_{k-i-1, r-1}(x, u).
        \end{equation*}
  \item Make the substitution (with relevant term replacement) from Proposition~\ref{prop:k_t:enum}
  \begin{equation*}
    [x^n]S_{t}(x, u)\big|_{u=1} = \frac{t+1}{(k+1)n+t+1}\binom{(k+1)n+t+1}{n}
  \end{equation*}
  and use the identity \eqref{eq:FC-summation} to simplify the sum with bounds $0 \leq i \leq t$.
\end{itemize}
Therefore
\begin{align}
 \partial_u S_{t,r}(x, u)\big|_{u=1}
  & = xS_t(x, 1)\sum_{i=0}^t\partial_u S_{k-i-1,r-1}(x, u)\big|_{u=1}\notag \\
  & \quad + x\sum_{j=0}^{r-1}x^{j}\partial_u S_{t,r-1-j}(x, u)\big|_{u=1}\frac{t+1}{j+1}\binom{(k+1)j+k-t-1}{j} \label{eq:base-for-app}\\
  & \quad + x^{r}S_t(x, 1)\sum_{i=0}^t\frac{(k-i)^2}{(k+1)(r-1)+k-i}\binom{(k+1)(r-1)+k-i}{r-1}. \notag
\end{align}
\begin{theorem}\label{thm:gen-func-down-steps}
  For $r \geq 2$, the generating function for $S_{t,r}(x, u)$ after differentiating with respect to $u$ and setting $u =1$ is
  \begin{align*}
    \partial_u S_{t,r}(x, u)\big|_{u=1}
    & = \frac{1-(1-z)^{t+1}}{z(1-z)^t} - \sum_{j=1}^rz^{j-1}(1-z)^{kj}\frac{t+1}{(k+1)j+t+1}\binom{(k+1)j+t+1}{j}\\
    & \quad + \sum_{j = 1}^{r-1}z^j(1-z)^{jk}\frac{t+1}{(k+1)j+t+1}\binom{(k+1)j+t+1}{j}\\
    & \quad - z^r(1-z)^{rk}\frac{t(t+1)}{(k+1)r+t+1}\binom{(k+1)r+t+1}{r}.
  \end{align*}
\end{theorem}
The key to the proof of the above theorem is the following identity, which follows from the symbolic
decomposition of $k_t$-Dyck paths given in \eqref{eq:symbolic-decomposition}. The remainder of the proof
of Theorem~\ref{thm:gen-func-down-steps} consists of algebraic manipulations and is detailed in Appendix~\ref{ap:proof}.
\begin{proposition}\label{prop:gen-func-simp}
  For $n \geq 0$, and $0 \leq t \leq k$ we have that
  \begin{align*}
    &\sum_{j=0}^{n-1}\frac{t+1}{(k+1)(n-1-j)+t+1}\binom{(k+1)(n-1-j)+t+1}{n-1-j}\frac{t+1}{j+1}\binom{(k+1)j+k-t-1}{j}\\
    & \quad = \frac{t+1}{(k+1)n+t+1}\binom{(k+1)n+t+1}{n}.
  \end{align*}
\end{proposition}
This result can be seen as a generalized Chu--Vandermonde-type
identity, and is similar (though not equivalent) to the
family of identities given in~\cite{Rothe:1793:serierum-reversione}
and mentioned in~\cite{Gould:1999:power-sum}, known as Rothe--Hagen
identities. In particular, it can be proved algebraically using these identities, but we still provide a bijective argument here, 
as we are not aware of this combinatorial interpretation in the literature.

\begin{proof}[of Proposition~\ref{prop:gen-func-simp}]
  We enumerate $k_t$-Dyck paths in two ways to obtain this identity. We have from Proposition~\ref{prop:k_t:enum}
  that the number of $k_t$-Dyck paths of length $(k+1)n$ is
  \begin{equation*}
    \frac{t+1}{(k+1)n + t + 1}\binom{(k+1)n + t + 1}{n}.
  \end{equation*}
  However, the symbolic decomposition of $k_t$-Dyck paths given in Figure~\ref{fig:kDyckdec} shows that
  we can decompose a $k_t$-Dyck path of length $(k+1)n$ into a $k_{k-i-1}$-Dyck path of
  length $(k+1)j$ (where $0 \leq i \leq t$ and $0 \leq j \leq n-1$) and a $k_t$-Dyck path of length $(k+1)(n-1-j)$.
  Using identity \eqref{eq:FC-summation} to simplify the double sum results in
  \begin{align*}
    & \sum_{j=0}^{n-1}\Bigg[\sum_{i=0}^t\frac{k-i}{(k+1)j+k-i}\binom{(k+1)j+k-i}{j}\Bigg]\\
    & \qquad \qquad \qquad \qquad \qquad \times \frac{t+1}{(k+1)(n-1-j)+t+1}\binom{(k+1)(n-1-j)+t+1}{n-1-j}\\
    & = \sum_{j=0}^{n-1}\frac{t+1}{j+1}\binom{(k+1)j+k-t-1}{j}\frac{t+1}{(k+1)(n-1-j)+t+1}\binom{(k+1)(n-1-j)+t+1}{n-1-j}.
  \end{align*}
  Since these are both valid ways to enumerate $k_t$-Dyck paths of length $(k+1)n$, the identity holds.
\end{proof}
\begin{corollary}\label{cor:downsteps-explicit-gf}
  The total number of down-steps between the $r$-th and the $(r+1)$-th
  up-steps in $k_t$-Dyck paths of length $(k+1)n$ for $r < n$ is
  \begin{align*}
    s_{n, t, r} 
    & = \frac{k}{n+1}\binom{(k+1)n}{n} - \frac{k-t}{n+1}\binom{(k+1)n+t}{n}\\
    & \quad + \sum_{j=1}^{r} \frac{t+1}{(k+1)j + t + 1}\frac{k}{n-j+1} \binom{(k+1)j + t +1}{j} \binom{(k+1)(n-j)}{n-j}
  \end{align*}
\end{corollary}
By using \eqref{eq:FC_telescope}, this result corresponds to that of
Corollary~\ref{cor:downsteps-explicit}. In fact, we can even rewrite
this result in the form
\begin{equation*}
  s_{n, t, r} = s_{n, t, 0} + \sum_{j=1}^{r} \frac{t+1}{(k+1)j + t + 1}\frac{k}{n-j+1} \binom{(k+1)j + t +1}{j} \binom{(k+1)(n-j)}{n-j}.
\end{equation*}
\begin{proof}[of Corollary~\ref{cor:downsteps-explicit-gf}]
  We begin by calculating $[x^{n}]z^{j-1}(1-z)^{kj}$ for $j \geq 1$, using Lemma \ref{lem:cauchy},
  \begin{equation*}
    [x^{n}]z^{j-1}(1-z)^{kj} = -\frac{k}{n-j+1}\binom{(k+1)(n-j)}{n-j}.
  \end{equation*}
  Additionally, since $\frac{1-(1-z)^{t+1}}{z(1-z)^{t}} = \frac{1}{z(1-z)^t} - \frac{1}{z} + 1$, in a similar way we have
  \begin{equation*}
    [x^{n}]\frac{1}{z(1-z)^{t}} = \binom{(k+1)n+t}{n}\frac{t-k}{n+1},
  \end{equation*}
  as well as
  \begin{equation*}
    [x^{n}]\frac{1}{z}  = -\frac{k}{n+1}\binom{(k+1)n}{n}.
  \end{equation*}
  Combining these with the result in Theorem~\ref{thm:gen-func-down-steps}, we have that for $r < n$
  \begin{align*}
    [x^{n}]\partial_u S_{t,r}(x, u) & = [x^{n}]\frac{1-(1-z)^{t+1}}{z(1-z)^t}\\
                                                            & \quad - \sum_{j=1}^r\frac{t+1}{(k+1)j+t+1}\binom{(k+1)j+t+1}{j}[x^{n}]z^{j-1}(1-z)^{kj}\\
                                                            & = \frac{k}{n+1}\binom{(k+1)n}{n} - \frac{k-t}{n+1}\binom{(k+1)n+t}{n}\\
                                                            & \quad + \sum_{j=1}^r\frac{t+1}{(k+1)j+t+1}\binom{(k+1)j+t+1}{j}\frac{k}{n-j+1}\binom{(k+1)(n-j)}{n-j}.\qedhere
  \end{align*}
\end{proof}

In this section we have shown that the generating function approach is
just as suitable as the bijective approach for the investigation of
the down-step parameter, with the added benefit of providing a
strategy for computing the variance for all fixed values of $r$. In
fact, combining these two approaches reveals interesting summation
identities like \eqref{eq:FC_telescope}.

\section{Asymptotics for down-step statistics}\label{sec:asy}

In this section we calculate asymptotics for the average number of down-steps as well as the variance 
in the number of down-steps in $k_t$-Dyck paths of length $(k+1)n$. These multivariable computations 
can become quite involved, and thus were performed\footnote{A
  worksheet containing these computations can be found at
  \url{https://gitlab.aau.at/behackl/kt-dyck-downstep-code}.}
using the asymptotic expansions
module~\cite{Hackl-Heuberger-Krenn:2016:asy-sagemath} in
SageMath~\cite{SageMath:2020:9.0}.

Formally, we conduct this analysis by considering the random variable $X_{n,t,r}$ which models the
number of down-steps between the $r$-th and the $(r+1)$-th up-steps in a $k_{t}$-Dyck path of
length $(k+1)n$ chosen uniformly at random.

\subsection{Expected Value}

From the explicit formulas for $s_{n,t,r}$ provided in Corollary~\ref{cor:downsteps-explicit} and
Corollary~\ref{cor:downsteps-explicit-gf} (for the case where $0\leq r < n$) as well as
Proposition~\ref{prop:last-downsteps} and Proposition~\ref{thm:down-steps-end} (for $r = n$) the expected
number of down-steps can be obtained immediately by the relation
\[ \E X_{n,t,r} = \frac{s_{n,t,r}}{C_{n,t}}. \]
From there, we can determine asymptotic expansions (for $n\to\infty$ and with $k$, $t$
fixed) of down-steps between the $r$-th and $(r+1)$-th up-steps.

An asymptotic expansion with error term $O(1/n)$ can be obtained
directly from the expressions for $s_{n, t, r}$ and $C_{n, t}$. Expansions of arbitrary precision can be obtained by
rewriting the products as factorials and applying Stirling's
formula, cf.~\DLMF{5.11}{3}; the first few terms are contained in the SageMath worksheet.

\begin{proposition}\label{prop:asy1}
  For $k_t$-Dyck paths of length $(k+1)n$, the following results hold:
\begin{enumerate}
\item  The expected number of down-steps before the first up-step is
  \begin{align*}
    \E X_{n,t,0}  & = \frac{k \ (kn+t+1)\cdots (kn+1)}{(n+1)(t+1) \ ((k+1)n+t)\cdots((k+1)n+1)} -
      \frac{(k-t)(kn+t+1)}{(n+1)(t+1)}\\
    & = \frac{k}{t+1}\bigg(\frac{k^{t+1}}{(k+1)^t}
      - k + t\bigg) + O\bigg(\frac{1}{n}\bigg).
  \end{align*}
\item  For fixed $1 \leq r < n$, the expected number of down-steps between the
  $r$-th and $(r+1)$-th up-steps is
  \begin{align*}
    \E X_{n,t,r}
    & = \frac{k \ (kn+t+1)\cdots (kn+1)}{(n+1)(t+1) \ ((k+1)n+t)\cdots((k+1)n+1)} -
      \frac{(k-t)(kn+t+1)}{(n+1)(t+1)}\\
    & \quad + \sum_{j=1}^{r}\frac{t+1}{(k+1)j+t+1} \binom{(k+1)j+t+1}{j}\\
    & \qquad \qquad \qquad\times \frac{k\ n  \cdots (n-j+2) \ (kn+t+1) \cdots (k(n-j)+1)}{(t+1) \ ((k+1)n + t)\cdots ((k+1)(n - j) + 1)}\\
    & = \frac{k^{t+2}}{(k+1)^{t}}
    \sum_{j=1}^{r} \frac{1}{(k+1)j+t+1} \binom{(k+1)j+t+1}{j}
    \bigg(\frac{k^{k}}{(k+1)^{k+1}}\bigg)^{j} \\
    & \qquad + \frac{k}{t+1}\bigg(\frac{k^{t+1}}{(k+1)^t} - k + t\bigg)      
     + O\bigg(\frac{1}{n}\bigg).
  \end{align*}
\end{enumerate}
\end{proposition}

Analogously, we obtain the following result for $\E X_{n,t,n}$ as well as $\E X_{n,t,n-r}$.

\begin{proposition}\label{prop:asy2}
  For $k_t$-Dyck paths of length $(k+1)n$ we find:
\begin{enumerate}
\item The expected number of down-steps after the last up-step is
  \begin{align*}
    \E X_{n,t,n} & = \frac{((k+1)(n+1)+t)\cdots ((k+1)n+t+1)}{(n+1) \ (k(n+1)+t+1)\cdots(kn+t+2)} - (t+1)\\
    & = \frac{(k+1)^{k+1}}{k^k} - (t+1) + O\bigg(\frac{1}{n}\bigg).
  \end{align*}
\item\label{itm:prop:asy2:n-r} For fixed $1 \leq r \leq n$, the expected number of down-steps between the $(n-r)$-th and the $(n-r+1)$-th up-steps is
  \begin{align*}
    \E X_{n,t,n-r} & = \frac{((k+1)(n+1)+t)\cdots ((k+1)n+t+1)}{(n+1) \ (k(n+1)+t+1)\cdots(kn+t+2)} - (k+1) \\
    & \quad - \sum_{j = 2}^{r} \binom{(j-1)(k+1)}{j}
    \frac{n \cdots (n-j+2) \ (kn+t+1) \cdots (k(n-j+1)+t+2)}
    {(j-1)((k+1)n+t) \cdots ( (k+1)(n-j+1)+t+1 )}\\
    & = \frac{(k+1)^{k+1}}{k^k} - k-1
    - \sum_{j = 2}^{r}  \frac{1}{j-1}  \binom{(j-1)(k+1)}{j}
    \frac{k^{k(j-1)}}{(k+1)^{(k+1)(j-1)}} + O\bigg(\frac{1}{n}\bigg).
  \end{align*}
  \end{enumerate}
\end{proposition}

\begin{remark}
  Note that extracting the asymptotic behaviour from the explicit formulas is,
  of course, not the only viable strategy: alternatively, we could
  extract the growth directly from the generating functions by investigating
  the substitution $x = z (1 - z)^k$ from \eqref{eq:substitution}, where the variable $x$ symbolically
  corresponds to the number of up-steps in $k_t$-Dyck paths.

  To be more precise, instead of using Lemma~\ref{lem:cauchy} for translating
  between the ``generating function world'' and the ``coefficient world'',
  we could also observe that $\varphi(z) := z(1-z)^k$ has an analytic inverse
  as long as $\varphi'(z) \neq 0$. And indeed: $z_0 = 1/(k+1)$ is the zero
  of smallest modulus of the derivative $\varphi'(z)$.
  Singular inversion~\cite[Chapter VI.7]{Flajolet-Sedgewick:ta:analy}
  (which basically expands $\varphi(z)$ as a power series in $z-z_0$ to
  find that the inverse function has a square root singularity at
  $x_0 = \varphi(z_0) = \frac{k^k}{(k+1)^{k+1}}$)
  then allows to obtain the asymptotic growth of the coefficients of
  the inverse function of $\varphi(z)$ (or compositions thereof,
  which is what obtained in the formulas throughout
  Section~\ref{sec:generating-functions}) by means of singularity analysis.
  These calculations can also be carried out with SageMath's
  module on asymptotic expansions~\cite{Hackl-Heuberger-Krenn:2016:asy-sagemath}.
\end{remark}

It is worthwhile to take a closer look at the asymptotic expansions given in
Propositions~\ref{prop:asy1} and~\ref{prop:asy2}. In particular, if we divide these expansions by
$k$ and let $k \to \infty$, a very particular behaviour (illustrated in
Figures~\ref{fig:downsteps-k-asy}, \ref{fig:asymptotic-plots:a}, \ref{fig:asymptotic-plots:b}) arises.
Note that dividing by $k$ is fairly intuitive as $k_t$-Dyck
paths need to have $k$ down-steps for every up-step, normalization thus yields a parameter
for down-steps that is on the same scale as the number of up-steps.

For the \textit{beginning of paths} (Proposition~\ref{prop:asy1}), the limiting behaviour depends
on $t$.  In particular, for $t=0$ we obtain
\begin{equation}\label{eq:lamb1}
\lim_{k \to \infty} \lim_{n \to \infty}
\frac{\E X_{n,0,r}}{k} =
\sum_{j=1}^r \frac{j^{j-1}}{j!} \frac{1}{e^j},
\end{equation}
and for $t=k$,
\begin{equation}\label{eq:lamb11}
\lim_{k \to \infty} \lim_{n \to \infty}
\frac{\E X_{n,k,r}}{k} =
\sum_{j=1}^{r+1} \frac{j^{j-1}}{j!} \frac{1}{e^{j}}.
\end{equation}
Notice that~\eqref{eq:lamb11} is the same sum as in~\eqref{eq:lamb1} with a shifted upper
limit of $r+1$ instead of $r$.
This means that for fixed $r$ and $k \to \infty$, the quantity
$\E X_{n,k,r} / k$ has the same asymptotic behaviour as
$\E X_{n,0,r+1} / k$. This is also indicated by Figure~\ref{fig:asymptotic-plots:c},
compare the values represented by the red dots (corresponding to $\E X_{n,k,r}/k$)
to the values represented by the blue dots one unit to the right
(corresponding to $E X_{n, 0, r+1}/k$).

In contrast, analogous limiting expressions
for the \textit{end of paths} (Proposition~\ref{prop:asy2})
do not depend on $t$, with the exception of down-steps after the last up-step.
Namely, for $k\to \infty$ we have
\begin{equation*}
\lim_{n \to \infty}
\frac{\E X_{n,t,n}}{k} = e - \frac{2t + 2 - e}{2k} + O\bigg(\frac{1}{k^2}\bigg).
\end{equation*}
Furthermore, for fixed $r \geq 1$ we have
\begin{equation}\label{eq:lamb2}
\lim_{k \to \infty} \lim_{n \to \infty}
\frac{\E X_{n,t,n-r}}{k} =
e - \sum_{j=1}^{r}\frac{(j-1)^{j-1}}{j!}\frac{1}{e^{j-1}}.
\end{equation}
Together with the observation that the asymptotic main term
in $n$ given in Proposition~\ref{prop:asy2}~(\ref{itm:prop:asy2:n-r})
does not depend on $t$,
these considerations explain why the rightmost points in
Figure~\ref{fig:asymptotic-plots:d} (i.e.,
those corresponding to $r=n$) are evenly distributed, and the points for $r<n$ nearly coincide for
$0 \leq t \leq k$.

The expressions in~\eqref{eq:lamb1} and~\eqref{eq:lamb2} are related to $T(x)$, the
\textit{Cayley tree function}\footnote{The Cayley tree function $T(x)$ is
  also strongly related to the \emph{Lambert W function} $W(x)$ via $-W(-x) = T(x)$.}
which satisfies the equation $T(x) = x \exp(T(x))$ (see \cite[VI.16]{Flajolet-Sedgewick:ta:analy}), as follows.
We have $T(x) = \sum_{j \geq 1}   \frac{j^{j-1}}{j!} x^j $
and $\frac{1}{T(x)} = \frac{1}{x} -\sum_{j\geq 0}\frac{j^{j}}{(j+1)!}x^j$
(see \OEIS{A000169} in the OEIS~\cite{OEIS:2022}).
The right-hand side expressions in~\eqref{eq:lamb1} and~\eqref{eq:lamb2} are partial sums
of these power series, upon substitution of $x=1/e$.
Since we have $T(1/e) = 1$,
and since the sequence $(s_{n,t,r})_{0 \leq r \leq n}$ is increasing (strongly, except for $t=k$, $r=n-1$),
for each $0< \beta < 1$ we have
\[ 
  \lim_{k \to \infty} \lim_{n \to \infty} \frac{s_{n,t,[\beta n]}}{kC_{n,t}} = 1.
\]
This observation is illustrated in Figure~\ref{fig:downsteps-k-asy}.

\begin{figure}[ht]
  \centering
  \pgfplotstabletranspose[input colnames to={x1}, colnames from={x1}]{\datatableA}{figures/downsteps-k-asy-1.dat}
  \pgfplotstabletranspose[input colnames to={x2}, colnames from={x2}]{\datatableB}{figures/downsteps-k-asy-2.dat}
  \pgfplotstabletranspose[input colnames to={x3}, colnames from={x3}]{\datatableC}{figures/downsteps-k-asy-3.dat}
  \pgfplotstabletranspose[input colnames to={x4}, colnames from={x4}]{\datatableD}{figures/downsteps-k-asy-4.dat}
  \begin{tikzpicture}[scale=0.8]
    \begin{axis}[
      xlabel = $r/n$,
      xtick distance = {0.2},
      ytick = {0, 0.5, 1, 1.5, 2, 2.5, 3},
      x = 8cm,
      legend pos = south east,
      tick label style={/pgf/number format/1000 sep=}]
      \addplot [gray!33, line width=1pt] table[x=x1, y=y1] {\datatableA};
      \addplot [gray!66, line width=1pt] table[x=x2, y=y2] {\datatableB};
      \addplot [gray, line width=1pt] table[x=x3, y=y3] {\datatableC};
      \addplot [black, line width=1pt] table[x=x4, y=y4] {\datatableD};
      \draw[help lines, blue] (-0.1,1) -- (1.1,1);
    \end{axis}
  \end{tikzpicture}
  \caption{Distribution of $X_{n, t, r}/k$ for $t = 7$ and several combinations
    of values for $n$ and $k$:
    $(n, k) \in \{(10, 100), (40, 400), (160, 1600), (640, 6400)\}$. Darker lines correspond to larger values of $n$ and
  $k$.}
  \label{fig:downsteps-k-asy}
\end{figure}
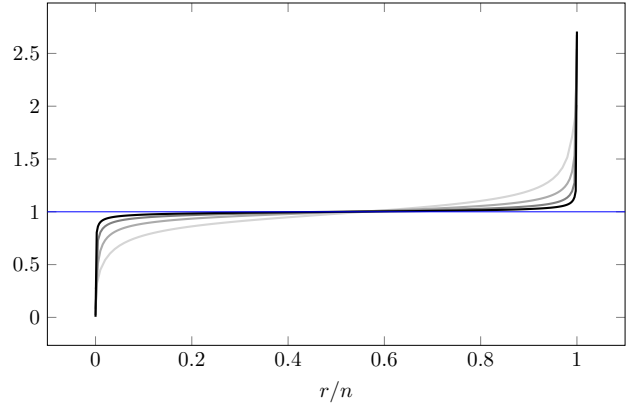

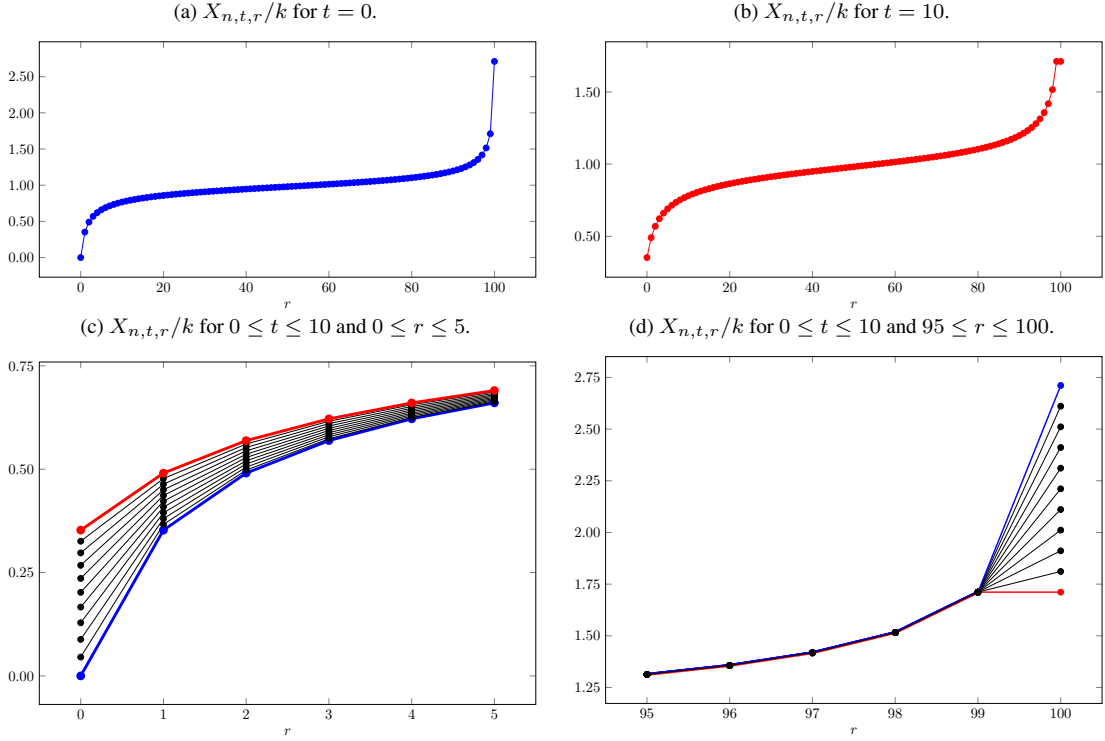
\begin{figure}[ht]
  \begin{subfigure}[t]{0.49\linewidth}
    \subcaption{$X_{n,t,r}/k$ for $t = 0$.}
    \label{fig:asymptotic-plots:a}
  \begin{tikzpicture}[scale = 0.55]
  \begin{axis}[
    xlabel = $r$,
    xtick distance = {20},
    ytick = {0, 0.5, 1, 1.5, 2, 2.5, 3},
    y tick label style = {/pgf/number format/.cd, fixed, fixed zerofill, precision=2, /tikz/.cd},
    x = 0.1cm,
    legend pos = south east,
    tick label style={/pgf/number format/1000 sep=}]
    \addplot [blue, mark = *] coordinates {(0, 0.000000000000000)
(1, 0.352608887565825) (2, 0.490625703413219) (3, 0.569215270679994) (4,
0.621840323395299) (5, 0.660435990656537) (6, 0.690444403297976) (7,
0.714745318225499) (8, 0.735023974506128) (9, 0.752341149448238) (10,
0.767402539956564) (11, 0.780698490762173) (12, 0.792581890204947) (13,
0.803314169016425) (14, 0.813093773217781) (15, 0.822074494623052) (16,
0.830377671106902) (17, 0.838100542100905) (18, 0.845322114677882) (19,
0.852107372459068) (20, 0.858510354131074) (21, 0.864576444068426) (22,
0.870344103099269) (23, 0.875846194509207) (24, 0.881111012806996) (25,
0.886163091098621) (26, 0.891023841421072) (27, 0.895712067548633) (28,
0.900244379378995) (29, 0.904635530603888) (30, 0.908898696032138) (31,
0.913045701038251) (32, 0.917087212734667) (33, 0.921032900321260) (34,
0.924891570450184) (35, 0.928671282216118) (36, 0.932379445440525) (37,
0.936022905191124) (38, 0.939608014911667) (39, 0.943140700093576) (40,
0.946626514071444) (41, 0.950070687247414) (42, 0.953478170829038) (43,
0.956853675989122) (44, 0.960201709215164) (45, 0.963526604503089) (46,
0.966832552959754) (47, 0.970123630306873) (48, 0.973403822722482) (49,
0.976677051412329) (50, 0.979947196270788) (51, 0.983218118967820) (52,
0.986493685784034) (53, 0.989777790509655) (54, 0.993074377724771) (55,
0.996387466787812) (56, 0.999721176877057) (57, 1.00307975345689) (58,
1.00646759657757) (59, 1.00988929146596) (60, 1.01334964192708) (61,
1.01685370715539) (62, 1.02040684265323) (63, 1.02401474607796) (64,
1.02768350899320) (65, 1.03141967569271) (66, 1.03523031050730) (67,
1.03912307530957) (68, 1.04310631931633) (69, 1.04718918377830) (70,
1.05138172477352) (71, 1.05569505812860) (72, 1.06014153154083) (73,
1.06473493034687) (74, 1.06949072519677) (75, 1.07442637230786) (76,
1.07956168022650) (77, 1.08491926145349) (78, 1.09052509338688) (79,
1.09640922154126) (80, 1.10260665002957) (81, 1.10915848155663) (82,
1.11611339436027) (83, 1.12352958094657) (84, 1.13147733013741) (85,
1.14004252168918) (86, 1.14933144191667) (87, 1.15947755559050) (88,
1.17065125056991) (89, 1.18307423490849) (90, 1.19704146729291) (91,
1.21295577998395) (92, 1.23138491712498) (93, 1.25316046844430) (94,
1.27956076467280) (95, 1.31267755172345) (96, 1.35623422906971) (97,
1.41770567741071) (98, 1.51520379502768) (99, 1.71089994857077) (100, 2.71089994857077)};
  \end{axis}
\end{tikzpicture}
\end{subfigure}\hfill
\begin{subfigure}[t]{0.49\linewidth}
  \subcaption{$X_{n,t,r}/k$ for $t = 10$.}
  \label{fig:asymptotic-plots:b}
\begin{tikzpicture}[scale = 0.55]
  \begin{axis}[
    xlabel = $r$,
    xtick distance = {20},
    ytick = {0, 0.5, 1, 1.5, 2, 2.5, 3},
    y tick label style = {/pgf/number format/.cd, fixed, fixed zerofill, precision=2, /tikz/.cd},
    x = 0.1cm,
    legend pos = south east,
    tick label style={/pgf/number format/1000 sep=}]
    \addplot [red, mark = *] coordinates{(0, 0.352587828469249)
(1, 0.490575584497216) (2, 0.569136511387592) (3, 0.621734122873477) (4,
0.660303479285874) (5, 0.690286525782723) (6, 0.714562839857198) (7,
0.734817507716745) (8, 0.752111179685700) (9, 0.767149447011463) (10,
0.780422565573454) (11, 0.792283347952118) (12, 0.802993159318038) (13,
0.812750388104768) (14, 0.821708774798188) (15, 0.829989610873302) (16,
0.837690093245122) (17, 0.844889189521181) (18, 0.851651846237456) (19,
0.858032066825635) (20, 0.864075201778442) (21, 0.869819679028828) (22,
0.875298329621457) (23, 0.880539416187502) (24, 0.885567440059335) (25,
0.890403781368638) (26, 0.895067211634390) (27, 0.899574307942912) (28,
0.903939790420298) (29, 0.908176799361450) (30, 0.912297124485533) (31,
0.916311395913039) (32, 0.920229244315281) (33, 0.924059436071754) (34,
0.927809988042818) (35, 0.931488265623644) (36, 0.935101067017884) (37,
0.938654696103226) (38, 0.942155025817247) (39, 0.945607553642131) (40,
0.949017450489449) (41, 0.952389604065357) (42, 0.955728657619921) (43,
0.959039044842749) (44, 0.962325021553433) (45, 0.965590694744170) (46,
0.968840049459056) (47, 0.972076973936693) (48, 0.975305283397565) (49,
0.978528742822992) (50, 0.981751089047282) (51, 0.984976052467637) (52,
0.988207378667036) (53, 0.991448850243173) (54, 0.994704309141581) (55,
0.997977679803502) (56, 1.00127299345929) (57, 1.00459441392710) (58,
1.00794626531526) (59, 1.01133306207697) (60, 1.01475954192937) (61,
1.01823070222910) (62, 1.02175184049628) (63, 1.02532859990272) (64,
1.02896702069571) (65, 1.03267359872158) (66, 1.03645535245584) (67,
1.04031990025142) (68, 1.04427554990173) (69, 1.04833140310540) (70,
1.05249747804632) (71, 1.05678485411045) (72, 1.06120584380993) (73,
1.06577419835743) (74, 1.07050535514647) (75, 1.07541673780951) (76,
1.08052812277804) (77, 1.08586209069657) (78, 1.09144458713953) (79,
1.09730562558457) (80, 1.10348017762111) (81, 1.11000931263470) (82,
1.11694167439169) (83, 1.12433541935631) (84, 1.13226079823901) (85,
1.14080365000959) (86, 1.15007021677467) (87, 1.16019391474088) (88,
1.17134507766099) (89, 1.18374535241078) (90, 1.19768962743233) (91,
1.21358065304289) (92, 1.23198607625361) (93, 1.25373736988171) (94,
1.28011272218538) (95, 1.31320370487219) (96, 1.35673350980530) (97,
1.41817680354296) (98, 1.51564545058592) (99, 1.71131232242875) (100, 1.71131232242875)};
  \end{axis}
\end{tikzpicture}
\end{subfigure}

\begin{subfigure}[t]{0.49\linewidth}
  \subcaption{$X_{n, t, r}/k$ for $0 \leq t \leq 10$ and $0 \leq r \leq 5$.}
  \label{fig:asymptotic-plots:c}
\begin{tikzpicture}[scale = 0.55]
  \begin{axis}[
    xlabel = $r$,
    xtick distance = {1},
    ytick = {0, 0.25, 0.5, 0.75, 1},
    y tick label style = {/pgf/number format/.cd, fixed, fixed zerofill, precision=2, /tikz/.cd},
    x = 2cm,
    y = 10cm,
    legend pos = south east,
    tick label style={/pgf/number format/1000 sep=}]
    \addplot [blue, mark = *, line width=2pt] coordinates{(0, 0.000000000000000) (1, 0.352608887565825) (2, 0.490625703413219) (3,
0.569215270679994) (4, 0.621840323395299) (5, 0.660435990656537)};
    \addplot [black, mark = *] coordinates{(0, 0.0455040871934605) (1, 0.366406998493149) (2, 0.497722973145079)
(3, 0.573716086724906) (4, 0.625030122167428) (5, 0.662857273731785)};
    \addplot [black, mark = *] coordinates{(0, 0.0882883239292443) (1, 0.380362389295476) (2, 0.505077852951329)
(3, 0.578435886615187) (4, 0.628398469809528) (5, 0.665425762931568)};
    \addplot [black, mark = *] coordinates{(0, 0.128541348974362) (1, 0.394400244375718) (2, 0.512651911713355) (3,
0.583353548125986) (4, 0.631932514081944) (5, 0.668132999846066)};
    \addplot [black, mark = *] coordinates{(0, 0.166437721535676) (1, 0.408456009468977) (2, 0.520409490176674) (3,
0.588448836708860) (4, 0.635619716695081) (5, 0.670970635723138)};
    \addplot [black, mark = *] coordinates{(0, 0.202139025539133) (1, 0.422474227042210) (2, 0.528317611761405) (3,
0.593702431752164) (4, 0.639447884399042) (5, 0.673930455060540)};
    \addplot [black, mark = *] coordinates{(0, 0.235794884049011) (1, 0.436407495182463) (2, 0.536345880941704) (3,
0.599095942900970) (4, 0.643405194964574) (5, 0.677004396576576)};
    \addplot [black, mark = *] coordinates{(0, 0.267543891356122) (1, 0.450215537266258) (2, 0.544466372616436) (3,
0.604611917669391) (4, 0.647480218446374) (5, 0.680184571683512)};
    \addplot [black, mark = *] coordinates{(0, 0.297514469617591) (1, 0.463864370992435) (2, 0.552653515336076) (3,
0.610233841479178) (4, 0.651661934110997) (5, 0.683463280593056)};
    \addplot [black, mark = *] coordinates{(0, 0.325825656341200) (1, 0.477325566520273) (2, 0.560883970770330) (3,
0.615946131162358) (4, 0.655939743400073) (5, 0.686833026186372)};
    \addplot [red, mark = *, line width=2pt] coordinates{(0, 0.352587828469249) (1, 0.490575584497216) (2, 0.569136511387592) (3,
      0.621734122873477) (4, 0.660303479285874) (5, 0.690286525782723)};
  \end{axis}
\end{tikzpicture}
\end{subfigure}\hfill
\begin{subfigure}[t]{0.49\linewidth}
  \subcaption{$X_{n, t, r}/k$ for $0 \leq t \leq 10$ and $95 \leq r \leq 100$.}
  \label{fig:asymptotic-plots:d}
\begin{tikzpicture}[scale = 0.55]
  \begin{axis}[
    xlabel = $r$,
    xtick distance = {1},
    ytick = {1.25, 1.5, 1.75, 2, 2.25, 2.5, 2.75},
    y tick label style = {/pgf/number format/.cd, fixed, fixed zerofill, precision=2, /tikz/.cd},
    x = 2cm,
    y = 5cm,
    legend pos = south east,
    tick label style={/pgf/number format/1000 sep=}]
    \addplot [black, mark = *] coordinates{(95, 1.31271618114590)
      (96, 1.35627088455379) (97, 1.41774026502893) (98, 1.51523621838972)
      (99, 1.71093022180368) (100, 2.61093022180368)};
    \addplot [black, mark = *] coordinates{(95, 1.31275795022374)
      (96, 1.35631051945420) (97, 1.41777766414964) (98, 1.51527127746211)
      (99, 1.71096295605235) (100, 2.51096295605235)};
    \addplot [black, mark = *] coordinates{(95, 1.31280284698484)
      (96, 1.35635312244851) (97, 1.41781786412133) (98, 1.51530896228555)
      (99, 1.71099814203551) (100, 2.41099814203551)};
    \addplot [black, mark = *] coordinates{(95, 1.31285085951021)
      (96, 1.35639868226631) (97, 1.41786085434291) (98, 1.51534926294914)
      (99, 1.71103577051782) (100, 2.31103577051782)};
    \addplot [black, mark = *] coordinates{(95, 1.31290197593369)
      (96, 1.35644718768893) (97, 1.41790662426347) (98, 1.51539216959005)
      (99, 1.71107583230954) (100, 2.21107583230954)};
    \addplot [black, mark = *] coordinates{(95, 1.31295618444172)
      (96, 1.35649862754912) (97, 1.41795516338194) (98, 1.51543767239326)
      (99, 1.71111831826631) (100, 2.11111831826631)};
    \addplot [black, mark = *] coordinates{(95, 1.31301347327300)
      (96, 1.35655299073079) (97, 1.41800646124685) (98, 1.51548576159127)
      (99, 1.71116321928885) (100, 2.01116321928885)};
    \addplot [black, mark = *] coordinates{(95, 1.31307383071827)
      (96, 1.35661026616874) (97, 1.41806050745600) (98, 1.51553642746382)
      (99, 1.71121052632269) (100, 1.91121052632269)};
    \addplot [red, line width=2pt]
    coordinates{(95, 1.31320370487219)
      (96, 1.35673350980530) (97, 1.41817680354296) (98, 1.51564545058592)
      (99, 1.71131232242875)};
    \begin{scope}
      \clip (95, 1.31320370487219) -- (96, 1.35673350980530) -- (97,
      1.41817680354296) -- (98, 1.51564545058592) -- (99, 1.71131232242875)
      -- (95, 1.71131232242875) -- cycle;
      \addplot [blue, line width=2pt, postaction={clip, postaction={draw, blue}}]
    coordinates{(95, 1.31320370487219)
      (96, 1.35673350980530) (97, 1.41817680354296) (98, 1.51564545058592)
      (99, 1.71131232242875)};
    \end{scope}
    \addplot [red, line width=1pt] coordinates{(99, 1.71131232242875) (100, 1.71131232242875)};
    \addplot [blue, line width=1pt] coordinates{(99, 1.71089994857077) (100, 2.71089994857077)};
    \addplot [black, mark = *] coordinates{(95, 1.31313724512002)
      (96, 1.35667044284835) (97, 1.41811729165619) (98, 1.51558966033758)
      (99, 1.71126023035788) (100, 1.81126023035788)};
    \draw[draw=red, fill=red] (100, 1.71131232242875) circle [radius = 2pt];
    \draw[draw=blue, fill=blue] (100, 2.71089994857077) circle [radius = 2pt];
    \end{axis}
\end{tikzpicture}
\end{subfigure}
\caption{Distribution of $X_{n, t, r}/k$ for $n = 100$, $k = 10$, and varying values of $t$ and $r$.}
\label{fig:asymptotic-plots}
\end{figure}

\subsection{Variance}

Fueled by the generating function approach presented in Section~\ref{sec:generating-functions} we
are able to obtain asymptotic expansions for the variance $\V X_{n,t,r}$ as well. To do so, we use
the known asymptotic behaviour for the expected value $\E X_{n,t,r}$ as discussed by
Propositions~\ref{prop:asy1} and~\ref{prop:asy2} together with the explicit formulas for the second
factorial moment $\E (X_{n,t,r} (X_{n,t,r} - 1))$ that are obtained by normalizing the coefficients
extracted from the second partial derivatives of the corresponding bivariate generating functions
(Propositions~\ref{prop:r=n-2d}, \ref{prop:r=0-2d}, and~\ref{prop:r=1-2d}).
Combining these quantities by means of the well-known formula
\begin{equation*}
  \V X_{n,t,r} = \E (X_{n,t,r} (X_{n,t,r} - 1)) + \E X_{n,t,r} - (\E X_{n,t,r})^{2}
\end{equation*}
yields the variance of the number of down-steps between the $r$-th and $(r+1)$-th down-steps in
$k_{t}$-Dyck paths of length $(k+1)n$.

While the calculations are not complicated per se, the involved asymptotic expansions are rather
large which makes carrying out these calculations by hand a tedious task. Thus, we used
SageMath~\cite{SageMath:2020:9.0} and included our calculations in one of our associated
worksheets\footnote{Our computations are publicly available at
  \url{https://gitlab.aau.at/behackl/kt-dyck-downstep-code}.}. The following theorem summarizes our
findings.
\begin{theorem}\label{thm:downstep-variance}
  Let $k$ and $t$ be fixed integers with $k\geq 1$ and $0\leq t\leq k$, and consider
  $n\to\infty$. Then the variance of the random variables $X_{n,t,0}$, $X_{n,t,1}$, and $X_{n,t,n}$
  modeling the number of down-steps before the first up-step, between the first and second up-steps,
  and after the last up-step respectively, admit the asymptotic expansions
  \begin{equation}\label{eq:downstep-variance:r=0}
    \mathbb{V} X_{n,t,0} = \frac{-\alpha_{t,k}^{2}k^{4} + 2(t+2)\alpha_{t,k} k^{3} +
    ((2t+3)(t+1)\alpha_{t,k} + 3t^{2} + 4t) k^{2} + t(t+1)k}{(t+1)^{2}} + O\bigg(\frac{1}{n}\bigg),
  \end{equation}
  where $\alpha_{t,k} = \big(\frac{k}{k+1}\big)^{t} - 1$,
  \begin{align}\label{eq:downstep-variance:r=1}
    \begin{split}
      \V X_{n,t,1} &= \frac{(6k+t+3)k^{k+t+2}}{(k+1)^{k+t+1}} + \frac{(4k+1)k^{t+2}}{(k+1)^{t}(t+1)} - \frac{(4k^{2} - 2kt + 3k
        -t)k}{(t+1)}\\
      & \quad - \bigg(\frac{(t+1)k^{k+t+2} + k^{t+2}(k+1)^{k+1} -
        k(k-t)(k+1)^{k+t+1}}{(k+1)^{k+t+1}(t+1)}\bigg)^{2} + O\bigg(\frac{1}{n}\bigg)
    \end{split}
  \end{align}
  and finally,
  \begin{equation}\label{eq:downstep-variance:r=n}
    \V X_{n, t, n} = \frac{(k + 1)^{2k+2} - (2k+1)(k + 1)^{k+1} k^{k}}{k^{2k}} + O\bigg(\frac{1}{n}\bigg).
  \end{equation}
\end{theorem}

\begin{remark}
  Observe that the main contribution in the asymptotic expansions of both the expected value and the
  variance is of constant order and independent of $n$. This hints towards the parameter admitting
  a discrete limiting distribution; details are subject to further investigations.
\end{remark}

\section*{Conclusion}
In this article we began investigations into statistics concerning
down-steps in $k_t$-Dyck paths under the assumption $0\leq t \leq k$.
This led to new combinatorial identities, as well as new bijective
proofs for known combinatorial identities.  Furthermore, within the
study of applications of this statistic in coding theory, we found a
new \emph{Cycle Lemma}-type result which also led us to a novel
interpretation of Catalan numbers.

There are several open questions related to the down-step statistic
that we intend to address in further work. This includes
the more involved case of $t > k$, a closer investigation to extended
Kreweras walks as mentioned in the introduction, as well as further
study of the asymptotic behaviour.

\acknowledgements

The authors would like to thank Clemens Heuberger for comments on the full version of this paper,
as well as anonymous referees for their thorough and valuable reviews.

\bibliographystyle{alpha}
\bibliography{ktdyck}

\appendix

\section{Proof of Proposition~\ref{prop:r=n-2d}}\label{ap:App-A}

\begin{proof}
  The second derivative of the functional equation given in \eqref{eq:en-down-gf} is
  \begin{align*}
  \partial_u^2S_t(x, u) & = x\sum_{i=0}^t\partial_u^2S_{k-i-1}(x, u)u^{k-i} + 2x\sum_{i=0}^t\partial_uS_{k-i-1}(x, u)(k-i)u^{k-i-1}\\
                        & \quad + x\sum_{i=0}^tS_{k-i-1}(x, 1)\partial_u^2S_t(x, u)
                        + x\sum_{i=0}^tS_{k-i-1}(x, u)(k-i)(k-i-1)u^{k-i-2}.
  \end{align*}
  Simplifying this expression, setting $u = 1$, and applying relevant substitutions yields
  \begin{align}\label{eq:substitution-S2t}
    \partial_u^2S_t(x, u)\big|_{u=1}(1-z)^{t+1} & = z(1-z)^k\sum_{i=0}^t\partial_u^2S_{k-i-1}(x, u)\big|_{u=1}- z(1-z)^k\sum_{i=0}^t\frac{(k-i)^2}{(1-z)^{k-i}}\nonumber\\
    & \quad   + (2 - z(1-z)^k)\sum_{i=0}^t\frac{k-i}{(1-z)^{k-i}} - 2\sum_{i=0}^t(k-i).
  \end{align}
  We show that the second derivative at $u = 1$ is equal to
  \begin{equation}\label{eq:r=n-second-derivative}
    \partial_u^2S_t(x, u) \big|_{u=1} = \frac{2(1 - (1-z)^{t+1})}{z^2(1-z)^{2k+t+1}} - \frac{2(t+k+2)}{z(1-z)^{k+t+1}} + \frac{2(k+1)}{z(1-z)^{k}} + \frac{(t+1)(t+2)}{(1-z)^{t+1}}
  \end{equation}
  by substituting it into the occurrences of $\partial_u^2S_{k-i-1}(x, u)\big|_{u=1}$ on the right-hand side of \eqref{eq:substitution-S2t}, and simplifying this to obtain
  \begin{align*}
    \partial_u^2S_t(x, u)\big|_{u=1}(1-z)^{t+1} & = z(1-z)^k\sum_{i=0}^t\frac{2(1 - (1-z)^{k-i})}{z^2(1-z)^{3k-i}} - z(1-z)^k\sum_{i=0}^t\frac{2(2k-i+1)}{z(1-z)^{2k-i}}\\
    & \quad + z(1-z)^k\sum_{i=0}^t\frac{2(k+1)}{z(1-z)^{k}} + z(1-z)^k\sum_{i=0}^t\frac{(k-i)(k-i+1)}{(1-z)^{k-i}}\\
  & \quad  - z(1-z)^k\sum_{i=0}^t\frac{(k-i)^2}{(1-z)^{k-i}} + (2 - z(1-z)^k)\sum_{i=0}^t\frac{(k-i)}{(1-z)^{k-i}}\\
  &\quad - 2\sum_{i=0}^t(k-i)\\
     & = \frac{2(1-(1-z)^{t+1})}{z^2(1-z)^{2k}} - \frac{2(t+k+2)}{z(1-z)^{k}} + 2(k+1)\frac{(1-z)^{t+1}}{z(1-z)^k}\\
    & \quad + (t+2)(t+1),
  \end{align*}
  thus proving \eqref{eq:r=n-second-derivative}. We then extract coefficients from \eqref{eq:r=n-second-derivative} using
  Lemma~\ref{lem:cauchy} to obtain the result, where
  \begin{equation*}
    [x^{n}]\frac{1}{z^2(1-z)^{2k+t+1}} = \frac{t+1}{n+2}\binom{(k+1)(n+2)+t}{n+1}, \quad [x^{n}]\frac{1}{(1-z)^{t+1}} = \frac{t+1}{n}\binom{(k+1)n+t}{n-1},
  \end{equation*}
  and
  \begin{equation*}
    [x^{n}]\frac{1}{z(1-z)^{k+t+1}} = \frac{t+1}{n+1}\binom{(k+1)(n+1)+t}{n}.\qedhere
  \end{equation*}
\end{proof}

\section{Proof of Theorem~\ref{thm:gen-func-down-steps}}\label{ap:proof}
\fontsize{10}{10}\selectfont
\begin{proof}
To make the proof simpler, we compute each term of \eqref{eq:base-for-app} individually, starting with the first.
\begin{flalign*}
  xS_t&(x, 1)\sum_{i=0}^t\partial_u S_{k-i-1,r-1}(x, u)\big|_{u=1} = z(1-z)^k S_t(x, 1)\sum_{i=0}^t\frac{1-(1-z)^{k-i}}{z(1-z)^{k-i-1}}&\\
  & \;- z(1-z)^k S_t(x, 1)\sum_{i=0}^t\sum_{j=1}^{r-1}z^{j-1}(1-z)^{jk}\frac{k-i}{(k+1)j+k-i}\binom{(k+1)j+k-i}{j}\\
  & \; + z(1-z)^k S_t(x, 1)\sum_{i=0}^t\sum_{j = 1}^{r-2}z^j(1-z)^{jk}\frac{k-i}{(k+1)j+k-i}\binom{(k+1)j+k-i}{j}\\
  & \; - z(1-z)^k S_t(x, 1)\sum_{i=0}^tz^{r-1}(1-z)^{(r-1)k}\frac{(k-i-1)(k-i)}{(k+1)(r-1)+k-i}\binom{(k+1)(r-1)+k-i}{r-1}.
\end{flalign*}
Simplifying the sum that is a geometric series (and including its excess terms into the $j=0$ cases of the second
and third terms) as well as evaluating sums in the second and third sums
by using \eqref{eq:FC-summation}, we obtain
\begin{flalign*}
  xS_t&(x, 1)\sum_{i=0}^t\partial_u S_{k-i-1,r-1}(x, u)\big|_{u=1}  = \frac{1-(1-z)^{t+1}}{z(1-z)^{t}}\\
  & \quad - \Big(\frac{1}{z}S_t(x, 1) - S_t(x, 1)\Big)\sum_{j = 0}^{r-1}z^{j+1}(1-z)^{(j+1)k}\frac{t+1}{j+1}\binom{(k+1)j+k-t-1}{j}\\
  & \quad - z^{r}(1-z)^{rk}S_t(x, 1)\sum_{i=0}^t\frac{(k-i)^2}{(k+1)(r-1)+k-i}\binom{(k+1)(r-1)+k-i}{r-1}.
\end{flalign*}
Set $B_{j, k ,t} := \frac{t+1}{j+1}\binom{(k+1)j+k-t-1}{j}$, the second term from \eqref{eq:base-for-app} simplifies to
\begin{flalign*}
  & x\sum_{j=0}^{r-1}x^{j}\partial_u S_{t,r-1-j}(x, u)\big|_{u=1}\frac{t+1}{j+1}\binom{(k+1)j+k-t-1}{j}&\\
  & = \sum_{j=0}^{r-1}z^{j+1}(1-z)^{k(j+1)}\frac{1-(1-z)^{t+1}}{z(1-z)^t}\frac{t+1}{j+1}\binom{(k+1)j+k-t-1}{j}\\
  & \quad - \sum_{j=0}^{r-1}\sum_{\ell=1}^{r-1-j}z^{\ell+j}(1-z)^{k(\ell+j+1)}\frac{t+1}{(k+1)\ell+t+1}\binom{(k+1)\ell+t+1}{\ell}B_{j, k ,t}\\
  & \quad + \sum_{j=0}^{r-1}\sum_{\ell=1}^{r-2-j}z^{\ell+j+1}(1-z)^{k(\ell+j+1)}\frac{t+1}{(k+1)\ell+t+1}\binom{(k+1)\ell+t+1}{\ell}B_{j, k ,t}\\
  & \quad - \sum_{j=0}^{r-1}z^{r}(1-z)^{rk}\frac{t(t+1)}{(k+1)(r-1-j)+t+1}\binom{(k+1)(r-1-j)+t+1}{r-1-j}B_{j, k ,t}.
\end{flalign*}
This is further simplified by applying Proposition~\ref{prop:gen-func-simp} and rewriting $(1-(1-z)^t)/(z(1-z)^t)$
in terms of $S_t(x, 1)$,
\begin{flalign*}
  x\sum_{j=0}^{r-1}&x^{j}\partial_u S_{t,r-1-j}(x, u)\big|_{u=1}\frac{t+1}{j+1}\binom{(k+1)j+k-t-1}{j}&\\
 & = \Big(\frac{1}{z}S_t(x, 1) - S_t(x, 1) - \frac{1}{z} + 1\Big)\sum_{j=0}^{r-1}z^{j+1}(1-z)^{k(j+1)}B_{j, k ,t}\\
 & \quad - \frac{1}{z}\sum_{j=0}^{r-1}\sum_{\ell=1}^{r-1-j}z^{\ell+j+1}(1-z)^{k(\ell+j+1)}\frac{t+1}{(k+1)\ell+t+1}\binom{(k+1)\ell+t+1}{\ell}B_{j, k ,t}\\
 & \quad + \sum_{j=0}^{r-1}\sum_{\ell=1}^{r-2-j}z^{\ell+j+1}(1-z)^{k(\ell+j+1)}\frac{t+1}{(k+1)\ell+t+1}\binom{(k+1)\ell+t+1}{\ell}B_{j, k ,t}\\
 & \quad - z^{r}(1-z)^{rk}\frac{t(t+1)}{(k+1)r+t+1}\binom{(k+1)r+t+1}{r},
\end{flalign*}
and finally, including the terms for $\ell = 0$ into each of the sums,
\begin{flalign*}
  x\sum_{j=0}^{r-1}&x^{j}\partial_u S_{t,r-1-j}(x, u)\big|_{u=1}\frac{t+1}{j+1}\binom{(k+1)j+k-t-1}{j}&\\
 & = \Big(\frac{1}{z}S_t(x, 1) - S_t(x, 1)\Big)\sum_{j=0}^{r-1}z^{j+1}(1-z)^{k(j+1)}B_{j, k ,t}\\
 & \quad - \frac{1}{z}\sum_{j=0}^{r-1}\sum_{\ell=0}^{r-1-j}z^{\ell+j+1}(1-z)^{k(\ell+j+1)}\frac{t+1}{(k+1)\ell+t+1}\binom{(k+1)\ell+t+1}{\ell}B_{j, k ,t}\\
 & \quad + \sum_{j=0}^{r-1}\sum_{\ell=0}^{r-2-j}z^{\ell+j+1}(1-z)^{k(\ell+j+1)}\frac{t+1}{(k+1)\ell+t+1}\binom{(k+1)\ell+t+1}{\ell}B_{j, k ,t}\\
 & \quad - z^{r}(1-z)^{rk}\frac{t(t+1)}{(k+1)r+t+1}\binom{(k+1)r+t+1}{r}.
\end{flalign*}
Combining these terms and the final term of \eqref{eq:base-for-app} which has not been simplified gives us
\begin{flalign*}
  \partial_u &S_{t,r}(x, u)\big|_{u=1} = \frac{1-(1-z)^{t+1}}{z(1-z)^{t}}\\
  & \quad - \Big(\frac{1}{z}S_t(x, 1) - S_t(x, 1)\Big)\sum_{j = 0}^{r-1}z^{j+1}(1-z)^{(j+1)k}B_{j, k ,t}\\
  & \quad - z^{r}(1-z)^{rk}S_t(x, 1)\sum_{i=0}^t\frac{(k-i)^2}{(k+1)(r-1)+k-i}\binom{(k+1)(r-1)+k-i}{r-1}\\
  & \quad + \Big(\frac{1}{z}S_t(x, 1) - S_t(x, 1)\Big)\sum_{j=0}^{r-1}z^{j+1}(1-z)^{k(j+1)}B_{j, k ,t}\\
  & \quad - \frac{1}{z}\sum_{j=0}^{r-1}\sum_{\ell=0}^{r-1-j}z^{\ell+j+1}(1-z)^{k(\ell+j+1)}\frac{t+1}{(k+1)\ell+t+1}\binom{(k+1)\ell+t+1}{\ell}B_{j, k ,t}\\
  & \quad + \sum_{j=0}^{r-1}\sum_{\ell=0}^{r-2-j}z^{\ell+j+1}(1-z)^{k(\ell+j+1)}\frac{t+1}{(k+1)\ell+t+1}\binom{(k+1)\ell+t+1}{\ell}B_{j, k ,t}\\
  & \quad - z^{r}(1-z)^{rk}\frac{t(t+1)}{(k+1)r+t+1}\binom{(k+1)r+t+1}{r}\\
  & \quad + z^{r}(1-z)^{rk}S_t(x, 1)\sum_{i=0}^t\frac{(k-i)^2}{(k+1)(r-1)+k-i}\binom{(k+1)(r-1)+k-i}{r-1}.
\end{flalign*}
From this, there are some clear cancellations of factors and we re-index the double sums, which leads to
\begin{flalign*}
  \partial_u &S_{t,r}(x, u)\big|_{u=1} = \frac{1-(1-z)^{t+1}}{z(1-z)^{t}}\\
  & \quad - \frac{1}{z}\sum_{m=1}^{r}z^{m}(1-z)^{km}\sum_{j=0}^{m-1}\frac{t+1}{(k+1)(m-1-j)+t+1}\binom{(k+1)(m-1-j)+t+1}{m-1-j}B_{j, k ,t}\\
  & \quad + \sum_{m=1}^{r-1}z^{m}(1-z)^{km}\sum_{j=0}^{m-1}\frac{t+1}{(k+1)(m-1-j)+t+1}\binom{(k+1)(m-1-j)+t+1}{m-1-j}B_{j, k ,t}\\
  & \quad - z^{r}(1-z)^{rk}\frac{t(t+1)}{(k+1)r+t+1}\binom{(k+1)r+t+1}{r},
\end{flalign*}
and finally, once again using Proposition~\ref{prop:gen-func-simp} we have obtained an expression which is in the form
of Theorem~\ref{thm:gen-func-down-steps}, concluding the proof
\begin{flalign*}
  \partial_u S_{t,r}(x, u)\big|_{u=1} &= \frac{1-(1-z)^{t+1}}{z(1-z)^{t}} - z^{r}(1-z)^{kr}\frac{t(t+1)}{(k+1)r+t+1}\binom{(k+1)r+t+1}{r}\\
  & \quad - \sum_{m=1}^{r}z^{m-1}(1-z)^{km}\frac{t+1}{(k+1)m+t+1}\binom{(k+1)m+t+1}{m}\\ 
  & \quad + \sum_{m=1}^{r-1}z^{m}(1-z)^{km}\frac{t+1}{(k+1)m+t+1}\binom{(k+1)m+t+1}{m}.
\end{flalign*}
\end{proof}

\end{document}